\documentclass[draft]{amsart}
\usepackage{amssymb}
\usepackage{braket}
\usepackage[all]{xy}
\usepackage{mathrsfs}
\usepackage{braket}

\newtheorem{thm}{Theorem}[section]
\newtheorem{cor}[thm]{Corollary}
\newtheorem{prop}[thm]{Proposition}
\newtheorem{lem}[thm]{Lemma}
\newtheorem{claim}[thm]{Claim}
\theoremstyle{definition}
\newtheorem{dfn}[thm]{Definition}
\newtheorem{ex}[thm]{Example}
\newtheorem{fact}[thm]{Fact}

\theoremstyle{remark}
\newtheorem{rem}[thm]{Remark}

\newtheorem{notation}[thm]{Notation}

\newcommand{\Ob}{\mathrm{Ob}}
\newcommand{\ppr}{^{\prime}}

\newcommand{\Sett}{\mathit{Set}}
\newcommand{\Mon}{\mathit{Mon}}

\newcommand{\Ring}{\mathit{Ring}}

\newcommand{\TamG}{\mathit{Tam}(G)}

\newcommand{\Gs}{{}_G\mathit{set}}

\newcommand{\G}{\mathscr{G}}
\newcommand{\I}{\mathscr{I}}
\newcommand{\J}{\mathscr{J}}
\newcommand{\K}{\mathscr{K}}
\newcommand{\m}{\mathfrak{m}}
\newcommand{\p}{\mathfrak{p}}
\newcommand{\q}{\mathfrak{q}}
\newcommand{\Ker}{\mathit{Ker}}
\newcommand{\Spec}{\mathit{Spec}}
\newcommand{\pt}{\mathrm{pt}}
\newcommand{\id}{\mathrm{id}}
\newcommand{\Vo}{V_{\mathfrak{o}}}
\newcommand{\Ve}{V_{\mathfrak{e}}}
\newcommand{\Vt}{V_{\mathfrak{t}}}

\newcommand{\po}{\preceq}
\newcommand{\pn}{\prec}

\newcommand{\red}{_{\mathrm{red}}}

\newcommand{\TamM}{\mathit{Tam}^{\mathrm{MRC}}_{(G)}}
\newcommand{\TMRC}{T_{\mathrm{MRC}}}
\newcommand{\SMRC}{S_{\mathrm{MRC}}}

\numberwithin{equation}{section}

\begin{document}

\title[Ideals of Tambara functors]{Ideals of Tambara functors}

\author{Hiroyuki NAKAOKA}
\address{Department of Mathematics and Computer Science, Kagoshima University, 1-21-35 Korimoto, Kagoshima, 890-0065 Japan}

\email{nakaoka@sci.kagoshima-u.ac.jp}

\thanks{The author wishes to thank Professor Fumihito Oda for his suggestions and useful comments}
\thanks{The author wishes to thank Professor Serge Bouc and Professor Radu Stancu for the stimulating arguments and their useful comments and advices}
\thanks{Supported by JSPS Grant-in-Aid for Young Scientists (B) 22740005}

\begin{abstract}
For a finite group $G$, a Tambara functor on $G$ is regarded as a $G$-bivariant analog of a commutative ring. In this article, we consider a $G$-bivariant analog of the ideal theory for Tambara functors.
\end{abstract}

\maketitle

\tableofcontents

\section{Introduction and Preliminaries}

For a finite group $G$, a Tambara functor is regarded as a $G$-bivariant analog of a commutative ring, as seen in \cite{Yoshida}.
As such, for example a $G$-bivariant analog of the semigroup-ring construction was discussed in \cite{N_TamMack} and \cite{N_DressPolyHopf}, with relation to the Dress construction \cite{O-Y3}.
As part of this analogy, in this article we consider a $G$-bivariant analog of the ideal theory for Tambara functors.

In section \ref{section2}, we define an ideal of a Tambara functor, and show the fundamental theorem on homomorphisms for Tambara functors (Proposition \ref{PropKer}, Remark \ref{RemHom}).

In section \ref{section3}, we perform some operations on ideals (intersections, products, and sums), and show that an analog of the Chinese remainder theorem holds for any Tambara functor (Corollary \ref{CorChinese}).

In section \ref{section4}, we define a prime ideal of a Tambara functor, and construct the prime spectrum of a Tambara functor (Proposition \ref{PropSpec} and Theorem \ref{ThmMorphSpec}).
We may define an analog of an integral domain (resp. a field), by requiring the zero-ideal $(0)\subsetneq T$ is prime (resp. maximal).
We give a criterion for a Tambara functor to be a \lq field' (Theorem \ref{ThmField}), and show that the Burnside Tambara functor $\Omega$ is an \lq integral domain' (Theorem \ref{ThmOmegaDom}).

In section \ref{section4} we also show that any Tambara functor has a canonical ideal quotient, contained in a fixed point functor (Theorem \ref{ThmMRCUniv}). For example, this canonical quotient obtained from $\Omega$ becomes isomorphic to the fixed point functor associated to $\mathbb{Z}$ (Example \ref{ExMRCZ}).

\bigskip

Throughout this article, we fix a finite group $G$, whose unit element is denoted by $e$. Abbreviately we denote the trivial subgroup of $G$ by $e$, instead of $\{ e\}$.
$H\le G$ means $H$ is a subgroup of $G$.
$\Gs$ denotes the category of finite $G$-sets and $G$-equivariant maps.
If $H\le G$ and $g\in G$, 
then $H^g$ denotes the conjugate 
$H^g=g^{-1}Hg$.

A monoid is always assumed to be unitary and commutative. Similarly a ring is assumed to be commutative, with an additive unit $0$ and a multiplicative unit $1$.
We denote the category of monoids by $\Mon$, the category of rings by $\Ring$. 
A monoid homomorphism preserves units, and a ring homomorphism preserves $0$ and $1$.

For any category $\mathscr{C}$ and any pair of objects $X$ and $Y$ in $\mathscr{C}$, the set of morphisms from $X$ to $Y$ in $\mathscr{C}$ is denoted by $\mathscr{C}(X,Y)$. 

\bigskip

First we briefly recall the definition of a Tambara functor.

\begin{dfn}\label{DefTamFtr}
A {\it Tambara functor} $T$ {\it on} $G$ is a triplet $T=(T^{\ast},T_+,T_{\bullet})$ of two covariant functors
\[ T_+\colon\Gs\rightarrow\Sett,\ \ T_{\bullet}\colon\Gs\rightarrow\Sett \]
and one contravariant functor
\[ T^{\ast}\colon\Gs\rightarrow\Sett \]
which satisfies the following. Here $\Sett$ is the category of sets.
\begin{enumerate}
\item $T^{\alpha}=(T^{\ast},T_+)$ is a Mackey functor on $G$.
\item $T^{\mu}=(T^{\ast},T_{\bullet})$ is a semi-Mackey functor on $G$.

\noindent Since $T^{\alpha},T^{\mu}$ are semi-Mackey functors, we have $T^{\ast}(X)=T_+(X)=T_{\bullet}(X)$ for each $X\in\Ob(\Gs)$. We denote this by $T(X)$.
\item (Distributive law)
If we are given an exponential diagram
\[
\xy
(-12,6)*+{X}="0";
(-12,-6)*+{Y}="2";
(0,6)*+{A}="4";
(12,6)*+{Z}="6";
(12,-6)*+{B}="8";
(0,0)*+{exp}="10";
{\ar_{f} "0";"2"};
{\ar_{p} "4";"0"};
{\ar_{\lambda} "6";"4"};
{\ar^{\rho} "6";"8"};
{\ar^{q} "8";"2"};
\endxy
\]
in $\Gs$, then
\[
\xy
(-18,7)*+{T(X)}="0";
(-18,-7)*+{T(Y)}="2";
(0,7)*+{T(A)}="4";
(18,7)*+{T(Z)}="6";
(18,-7)*+{T(B)}="8";
{\ar_{T_{\bullet}(f)} "0";"2"};
{\ar_{T_+(p)} "4";"0"};
{\ar^{T^{\ast}(\lambda)} "4";"6"};
{\ar^{T_{\bullet}(\rho)} "6";"8"};
{\ar^{T_+(q)} "8";"2"};
{\ar@{}|\circlearrowright "0";"8"};
\endxy
\]
is commutative.
\end{enumerate}

If $T=(T^{\ast},T_+,T_{\bullet})$ is a Tambara functor, then $T(X)$ becomes a ring for each $X\in\Ob(\Gs)$, whose additive (resp. multiplicative) structure is induced from that on $T^{\alpha}(X)$ (resp. $T^{\mu}(X)$).
Those $T^{\ast}(f), T_+(f),T_{\bullet}(f)$ for morphisms $f$ in $\Gs$ are called {\it structure morphisms} of $T$. For each $f\in\Gs(X,Y)$,
\begin{itemize}
\item $T^{\ast}(f)\colon T(Y)\rightarrow T(X)$ is a ring homomorphism, called the {\it restriction} along $f$. 
\item $T_+(f)\colon T(X)\rightarrow T(Y)$ is an additive homomorphism, called the {\it additive transfer} along $f$.
\item $T_{\bullet}(f)\colon T(X)\rightarrow T(Y)$ is a multiplicative homomorphism, called the {\it multiplicative transfer} along $f$.
\end{itemize}
$T^{\ast}(f),T_+(f),T_{\bullet}(f)$ are often abbreviated to $f^{\ast},f_+,f_{\bullet}$.

A {\it morphism} of Tambara functors $\varphi\colon T\rightarrow S$ is a family of ring homomorphisms
\[  \varphi=\{\varphi_X\colon T(X)\rightarrow S(X) \}_{X\in\Ob(\Gs)}, \]
natural with respect to all of the contravariant and the covariant parts. We denote the category of Tambara functors by $\TamG$.
\end{dfn}

\smallskip

Remark that if $f\in\Gs(X,Y)$ is a $G$-map between transitive $G$-sets $X$ and $Y$, then the number of elements in a fiber of $f$
\[ \sharp f^{-1}(y)=\sharp\{ x\in X\mid f(x)=y \} \]
does not depend on $y\in Y$. This is called the {\it degree} of $f$, and denoted by $\deg f$.
\begin{dfn}
A Tambara functor $T$ is {\it additively cohomological} if $T^{\alpha}$ is cohomological as a Mackey functor $($\cite{B-B}$)$, namely, if for any $f\in\Gs(X,Y)$ between transitive $X,Y\in\Ob(\Gs)$,
\[ f_+f^{\ast}(b)=(\deg f)\, b\quad({}^{\forall}b\in T(Y)) \]
is satisfied.
\end{dfn}

In this article, a {\it Tambara functor} always means a Tambara functor on $G$.

\begin{ex}
$\ \ $
\begin{enumerate}
\item If we define $\Omega$ by
\[ \Omega(X)=K_0(\Gs/X) \]
for each $X\in\Ob(\Gs)$, where the right hand side is the Grothendieck ring of the category of finite $G$-sets over $X$, then $\Omega$ becomes a Tambara functor on $G$. This is called the {\it Burnside Tambara functor}.
For each $f\in\Gs(X,Y)$,
\[ f_{\bullet}\colon\Omega(X)\rightarrow\Omega(Y) \]
is the one determined by
\[ \qquad\quad f_{\bullet}(A\overset{p}{\rightarrow}X)=(\Pi_f(A)\overset{\pi}{\rightarrow}Y)\quad\  ({}^{\forall}(A\overset{p}{\rightarrow}X)\in\Ob(\Gs/X)), \]
where $\Pi_f(A)$ and $\pi$ is
\[ \Pi_f(A)=\Set{ (y,\sigma)| \begin{array}{c}y\in Y,\\ \sigma\colon f^{-1}(y)\rightarrow A\ \ \text{a map of sets},\\ p\circ\sigma=\id_{f^{-1}(y)}\end{array} }, \]

\[ \pi(y,\sigma)=y. \]
$G$ acts on $\Pi_f(A)$ by $g\cdot(y,\sigma)=(gy,{}^g\!\sigma)$, where ${}^g\!\sigma$ is the map defined by
\[ {}^g\!\sigma(x)=g\sigma(g^{-1}x)\quad({}^{\forall}x\in f^{-1}(gy)). \]

\item Let $R$ be a $G$-ring. If we define $\mathcal{P}_R$ by
\[ \mathcal{P}_R(X)=\{ G\text{-maps from}\ X\ \text{to}\ R \} \]
for each $X\in\Ob(\Gs)$, then $\mathcal{P}_R$ becomes a Tambara functor on $G$. This is called the {\it fixed point functor} associated to $R$.
\end{enumerate}
\end{ex}

For the properties of exponential diagrams, see \cite{Tam}. Here we only introduce the following.

\begin{rem}\label{RemExp}
Let $f\in\Gs(X, Y)$ be any morphism, and let $\nabla\colon X\amalg X\rightarrow X$ be the folding map. Then
\[
\xy
(-20,6)*+{X}="0";
(-20,-6)*+{Y}="2";
(-2,6)*+{X\amalg X}="4";
(24,6)*+{U\amalg U^{\prime}}="6";
(24,-6)*+{V}="8";
(0,0)*+{exp}="10";
{\ar_{f} "0";"2"};
{\ar_{\nabla} "4";"0"};
{\ar_{r\amalg r^{\prime}} "6";"4"};
{\ar^{t\cup t^{\prime}} "6";"8"};
{\ar^{s} "8";"2"};
\endxy
\]
where
\begin{eqnarray*}
V&=&\{(y,C)\mid y\in Y,\ C\subseteq f^{-1}(y)\},\\
U&=&\{(x,C)\mid x\in X,\ C\subseteq f^{-1}(f(x)),\ x\in C\},\\
U^{\prime}&=&\{(x,C)\mid x\in X,\ C\subseteq f^{-1}(f(x)),\ x\in \!\!\!\!\! /\ C\} ,
\end{eqnarray*}
\begin{eqnarray*}
r\colon U\rightarrow X&,&(x,C)\mapsto x,\\
r^{\prime}\colon U^{\prime}\rightarrow X&,&(x,C)\mapsto x,\\
t\colon U\rightarrow V&,&(x,C)\mapsto (f(x),C),\\
t^{\prime}\colon U^{\prime}\rightarrow V&,&(x,C)\mapsto (f(x),C),\\
s\colon V\rightarrow Y&,&(y,C)\mapsto y,
\end{eqnarray*}
is an exponential diagram.
\end{rem}

Any Tambara functor satisfies the following properties, which will be frequently used in this article.
\begin{fact}\label{FactTam}(\cite{Tam})
Let $T$ be a Tambara functor.
\begin{enumerate}
\item For any $G$-map $f\in\Gs(X,Y)$, if we let $\eta$ be the inclusion
\[ \eta\colon Y\setminus f(X)\hookrightarrow Y, \]
then we have $f_{\bullet}(0)=\eta_+(1)$.
\item (Projection formula)
For any $f\in\Gs(X,Y)$, $a\in T(X)$ and $b\in T(Y)$, we have
\[ f_+(af^{\ast}(b))=f_+(a)b. \]
\item (Addition formula)
For any $f\in\Gs(X,Y)$, in the notation of Remark \ref{RemExp}, we have
\[ f_{\bullet}(a+b)=s_+(t_{\bullet}r^{\ast}(a)\cdot t\ppr_{\bullet}r^{\prime\ast}(b)) \]
for any $a,b\in T(X)$.
\end{enumerate}
\end{fact}

\section{Definition and fundamental properties}
\label{section2}
\subsection{Definition of an ideal}

\begin{dfn}\label{DefIdeal}
Let $T$ be a Tambara functor. An {\it ideal} $\I$ of $T$ is a family of ideals $\I(X)\subseteq T(X)$ $({}^{\forall}X\in\Ob(\Gs))$ satisfying
\begin{enumerate}
\item[{\rm (i)}] $f^{\ast}(\I(Y))\subseteq \I(X)$,
\item[{\rm (ii)}] $f_+(\I(X))\subseteq \I(Y)$,
\item[{\rm (iii)}] $f_{\bullet}(\I(X))\subseteq f_{\bullet}(0)+\I(Y)$
\end{enumerate}
for any $f\in\Gs(X,Y)$. These conditions also imply
\[ \I(X_1\amalg X_2)\cong\I(X_1)\times\I(X_2) \]
for any $X_1,X_2\in\Ob(\Gs)$. 
\end{dfn}
Obviously when $G$ is trivial, this definition of an ideal agrees with the ordinary definition of an ideal of a commutative ring.

\begin{rem}
Let $T$ be a Tambara functor. For any ideal $\I\subseteq T$, we have $\I(\emptyset)=T(\emptyset)=0$.
\end{rem}

\begin{ex}
Trivial examples of ideals are the following.
\begin{enumerate}
\item $\I=T$ is an ideal of $T$. Often we except this trivial one.
\item If we let $\I(X)=0$ for each $X\in\Ob(\Gs)$, then $\I\subseteq T$ is an ideal. This is called the {\it zero ideal} and denoted by $(0)$.
\end{enumerate}
\end{ex}

\medskip

If $T$ is a Tambara functor, then $T^{\alpha}$ becomes a Green functor, with the cross product defined by
\[ \mathit{cr}\colon T^{\alpha}(X)\times T^{\alpha}(Y)\rightarrow T^{\alpha}(X\times Y)\ ;\ (x,y)\mapsto p_X^{\ast}(x)\cdot p_Y^{\ast}(y) \]
for any $X,Y\in\Ob(\Gs)$. (See \cite{Tam_Manu} or \cite{N_TamMack}.)

With this underlying Green functor structure, an ideal $\I\subseteq T$ gives a {\it functorial ideal} $\I\subseteq T^{\alpha}$. A functorial ideal is originally defined in \cite{Thevenaz}, and written in the following form using the cross product in \cite{Bouc}.
\begin{dfn}\label{DefFunctIdeal}
Let $A$ be a Green functor on $G$. A {\it functorial ideal} $\I$ of $A$ is a Mackey subfunctor $\I\subseteq A$, satisfying
\begin{eqnarray*}
\mathit{cr}(\I(X)\times A(Y))\subseteq \I(X\times Y)\\
\mathit{cr}(A(X)\times \I(Y))\subseteq \I(X\times Y)
\end{eqnarray*}
for any $X,Y\in\Ob(\Gs)$.
\end{dfn}
It is shown that $A/\I=\{ A(X)/\I(X) \}_{X\in\Ob(\Gs)}$ naturally becomes a Green functor, with structure morphisms and cross product induced from those of $A$ (\cite{Thevenaz}, \cite{Bouc}).

\begin{lem}\label{LemQuot}
Let $T$ be a Tambara functor, and let $\I\subseteq T$ be an ideal. Then $\I\subseteq T^{\alpha}$ becomes a functorial ideal.
\end{lem}
\begin{proof}
Since $T^{\alpha}$ is commutative as a Green functor, it suffices to show
\[ \mathit{cr}(T^{\alpha}(X)\times\I(Y))\subseteq\I(X\times Y) \]
for each $X,Y\in\Ob(\Gs)$. This immediately follows from
\[ p_X^{\ast}(x)\cdot p_Y^{\ast}(y)\in p_X^{\ast}\cdot \I(X\times Y)\subseteq\I(X\times Y) \]
for any $x\in T^{\alpha}(X)=T(X)$ and $y\in\I(Y)$.
\end{proof}

\begin{prop}\label{PropQuot}
Let $\I$ be an ideal of a Tambara functor $T$.
Then
\[ T/\I=\{ T(X)/\I(X)\}_{X\in\Ob(\Gs)} \]
has a natural structure of a Tambara functor induced from that of $T$.

Moreover, projections $T(X)\rightarrow T(X)/\I(X)$ give a morphism of Tambara functors $T\rightarrow T/\I$.
\end{prop}
\begin{proof}
By Lemma \ref{LemQuot} $T/\I$ becomes a Green functor, with induced restrictions and additive transfers.
Thus it suffices to show that multiplicative transfers are well-defined by
\[
f_{\bullet}\colon T(X)/\I(X)\rightarrow T(Y)/\I(Y)\ ;\ 
x+\I(X)\mapsto f_{\bullet}(x)+\I(Y),
\]
for each $f\in\Gs(X,Y)$.
This well-definedness follows from the addition formula. In fact, in the notation of Remark \ref{RemExp}, we have
\begin{eqnarray*}
f_{\bullet}(x+\I(X))&=&s_+(t_{\bullet}r^{\ast}(x)\cdot t\ppr_{\bullet}r^{\prime\ast}(\I(X)))\\
&\subseteq&s_+(t_{\bullet}r^{\ast}(x)\cdot t\ppr_{\bullet}(\I(U\ppr)))\\
&\subseteq&s_+(t_{\bullet}r^{\ast}(x)\cdot (t\ppr_{\bullet}(0)+\I(V)))\\
&\subseteq&s_+((t_{\bullet}r^{\ast}(x)\cdot t\ppr_{\bullet}(0))+\I(V))\\
&\subseteq&s_+(t_{\bullet}r^{\ast}(x)\cdot t\ppr_{\bullet}(0))+\I(Y)\\
&=&f_{\bullet}(x)+\I(Y).
\end{eqnarray*}

The latter part is trivial.
\end{proof}

\begin{ex}
Let $m$ be an arbitrary integer.
Let $\Omega$ be the Burnside Tambara functor, and let $m\Omega\subseteq\Omega$ be a family of ideals
\[ m\Omega=\{ m(\Omega(X))\}_{X\in\Ob(\Gs)}. \]
Then $m\Omega\subseteq\Omega^{\alpha}$ is a functorial ideal in the sense of Definition \ref{DefFunctIdeal}, while it is not necessarily an ideal of $\Omega$ in the sense of Definition \ref{DefIdeal}, except $m=0,\pm 1$.

In fact, when $G=\mathbb{Z}/2\mathbb{Z}$, for the canonical projection $p^G_e\colon G/e\rightarrow G/G$, we have
\[ (p^G_e)_{\bullet}(2(G/e\overset{\id}{\rightarrow}G/e))=2(G/G\overset{\id}{\rightarrow}G/G)+(G/e\overset{p^G_e}{\rightarrow}G/G), \]
and thus $(p^G_e)_{\bullet}(2\Omega(G/e))\subseteq (p^G_e)_{\bullet}(0)+2\Omega(G/G)$ is not satisfied since $(p^G_e)_{\bullet}(0)=0$.
\end{ex}

\begin{dfn}\label{DefShriek}
Let $T$ be a Tambara functor. For each $f\in\Gs(X,Y)$, define $f_!\colon T(X)\rightarrow T(Y)$ by
\[ f_!(x)=f_{\bullet}(x)-f_{\bullet}(0) \]
for any $x\in T(X)$.
With this definition, the condition {\rm (iii)} in Definition \ref{DefIdeal} can be written as 
\begin{enumerate}
\item[{\rm (iii)$\ppr$}] $f_!(\I(X))\subseteq \I(Y)$.
\end{enumerate}
\end{dfn}

\begin{lem}\label{LemShriek}
Let $T$ be a Tambara functor. Then, we have the following for any $f\in\Gs(X,Y)$.
\begin{enumerate}
\item $f_!$ satisfies $f_!(x)f_!(y)=f_!(xy)$ for any $x,y\in T(X)$.
\item If $f$ is surjective, then we have $f_!=f_{\bullet}$.
\item If
\[
\xy
(-6,6)*+{X\ppr}="0";
(6,6)*+{Y\ppr}="2";
(-6,-6)*+{X}="4";
(6,-6)*+{Y}="6";
(0,0)*+{\square}="8";
{\ar^{f\ppr} "0";"2"};
{\ar_{\xi} "0";"4"};
{\ar^{\eta} "2";"6"};
{\ar_{f} "4";"6"};
\endxy
\]
is a pull-back diagram, then $f\ppr_!\, \xi^{\ast}=\eta^{\ast}f_!$ holds.
\item If
\[
\xy
(-14,6)*+{X}="0";
(-14,-6)*+{Y}="2";
(0,6)*+{A}="4";
(14,6)*+{Z}="6";
(14,-6)*+{\Pi}="8";
(0,0)*+{\mathit{exp}}="10";
{\ar_{f} "0";"2"};
{\ar_{p} "4";"0"};
{\ar_>>>>>{\lambda} "6";"4"};
{\ar^>>>>{\rho} "6";"8"};
{\ar^{\pi} "8";"2"};
\endxy
\]
is an exponential diagram, then $\pi_+\rho_!\,\lambda^{\ast}=f_!\, p_+$ holds.
\item If $S$ is another Tambara functor and $\varphi\colon T\rightarrow S$ is a morphism of Tambara functors, then 
$\varphi_Yf_!=f_!\,\varphi_X$
holds.
\end{enumerate}
\end{lem}
\begin{proof}
{\rm (1)} follows from
\begin{eqnarray*}
f_!(x)f_!(y)&=&(f_{\bullet}(x)-f_{\bullet}(0))(f_{\bullet}(y)-f_{\bullet}(0))\\
&=&f_{\bullet}(x)f_{\bullet}(y)-f_{\bullet}(x)f_{\bullet}(0)-f_{\bullet}(0)f_{\bullet}(y)+f_{\bullet}(0)f_{\bullet}(0)\\
&=&f_{\bullet}(x)f_{\bullet}(y)-f_{\bullet}(0).
\end{eqnarray*}
{\rm (2)} follows from the fact that $f_!(0)=0$ is satisfied for any surjective $f$ (Fact \ref{FactTam}).
{\rm (3)} follows from that, for any $x\in T(X)$,
\begin{eqnarray*}
f\ppr_!\, \xi^{\ast}(x)&=&f\ppr_{\bullet} \xi^{\ast}(x)-f\ppr_{\bullet}(0)\ =\ f\ppr_{\bullet} \xi^{\ast}(x)-f\ppr_{\bullet} \xi^{\ast}(0)\\
&=&\eta^{\ast}f_{\bullet}(x)-\eta^{\ast}f_{\bullet}(0)\ =\ \eta^{\ast}(f_{\bullet}(x)-f_{\bullet}(0))\ =\ \eta^{\ast}(f_!(x)).
\end{eqnarray*}
{\rm (4)} follows from that, for any $a\in T(A)$,
\begin{eqnarray*}
\pi_+\rho_!\lambda^{\ast}(a)&=&\pi_+(\rho_{\bullet}\lambda^{\ast}(a)-\rho_{\bullet}\lambda^{\ast}(0))\\
&=&\pi_+\rho_{\bullet}\lambda^{\ast}(a)-\pi_+\rho_{\bullet}\lambda^{\ast}(0)\\
&=&f_{\bullet}p_+(a)-f_{\bullet}p_+(0)\ =\ f_!\, p_+(a).
\end{eqnarray*}
{\rm (5)} follows from the additivity of $\varphi$.

\end{proof}

\begin{prop}\label{PropKer}
If $\varphi\colon T\rightarrow S$ is a morphism of Tambara functors, then $\Ker\,\varphi=\{\Ker(\varphi_X)\}_{X\in\Ob(\Gs)}$ is an ideal of $T$.
\end{prop}
\begin{proof}

Condition {\rm (iii)$\ppr$} in Definition \ref{DefShriek} follows from {\rm (5)} in Lemma \ref{LemShriek}. The other conditions are obviously satisfied.
\end{proof}

\begin{rem}\label{RemHom}
If $\varphi\colon T\rightarrow S$ is a morphism of Tambara functors, then it can be shown easily that $\mathit{Im}\,\varphi=\{\mathit{Im}(\varphi_X)\}_{X\in\Ob(\Gs)}$ is a Tambara subfunctor of $S$, and we have a natural isomorphism of Tambara functors
\[ T/\Ker\,\varphi\overset{\cong}{\longrightarrow}\mathit{Im}\,\varphi, \]
compatible with $\varphi$ and the projection $T\rightarrow T/\Ker\,\varphi$. 

Proposition \ref{PropQuot} and Proposition \ref{PropKer} mean that the ideals of a Tambara functor $T$ and projections from $T$ correspond essentially bijectively.
\end{rem}

In this article, we say a morphism of Tambara functors $\varphi\colon T\rightarrow S$ is {\it surjective} if $\varphi_X$ is surjective for any $X\in\Ob(\Gs)$.
\begin{prop}\label{Prop1to1}
$\ \ $
\begin{enumerate}
\item Let $\varphi\colon T\rightarrow S$ be a morphism in $\TamG$. If we define $\varphi^{-1}(\J)$ by
\[ (\varphi^{-1}(\J))(X)=\varphi_X^{-1}(\J(X))\quad({}^{\forall}X\in\Ob(\Gs)) \]for each ideal $\J\subseteq S$, then $\varphi^{-1}$ gives a map
\[ \varphi^{-1}\colon\{\,\text{ideals of}\,\ S\,\}\rightarrow\{\,\text{ideals of}\,\ T\,\}\ ;\ \J\mapsto\varphi^{-1}(\J), \]
which preserves inclusions of ideals.
\item Let $\varphi\colon T\rightarrow S$ be a surjective morphism of Tambara functors. If we define $\varphi({\I})$ by
\[ \varphi({\I})(X)=\varphi_X(\I(X))\quad({}^{\forall}X\in\Ob(\Gs)) \]
for each ideal $\I\subseteq T$ containing $\Ker\,\varphi$, then this gives a map
\[ \{\,\text{ideals of}\,\ T\ \text{containing}\ \Ker\,\varphi\,\}\rightarrow\{\,\text{ideals of}\,\ S\,\}\ ;\ \I\mapsto\varphi({\I}) \]
which preserves inclusions of ideals.
\item Let $\I$ be an ideal of $T$, and let $p\colon T\rightarrow T/\I$ be the projection. Then the maps in {\rm (1)} and {\rm (2)} give a bijection
\[ p^{-1}\colon \{\,\text{ideals of}\,\ T/\I\,\}\overset{1:1}{\longrightarrow} \{\,\text{ideals of}\,\ T\ \text{containing}\ \I\,\}.\]
\end{enumerate}
\end{prop}
\begin{proof}
These follow immediately from the definition.
\end{proof}

For any $X\in\Ob(\Gs)$, we denote the unique map from $X$ to the one-point set $G/G$ by $\pt_X\colon X\rightarrow G/G$.

\begin{prop}\label{PropTrivIdeal}
Let $\I\subseteq T$ be an ideal. Then the following are equivalent.
\begin{enumerate}
\item $\I=T$.
\item $1\in\I(X)$ for any $X\in\Ob(\Gs)$.
\item $1\in\I(X)$ for some $\emptyset\ne X\in\Ob(\Gs)$.
\end{enumerate}
\end{prop}
\begin{proof}
Obviously {\rm (1)} and {\rm (2)} are equivalent, and {\rm (2)} implies {\rm (3)}.
Thus it suffices to show that {\rm (3)} implies {\rm (2)}. 

Assume $\I$ satisfies $1\in\I(X)$ for some non-empty $X\in\Ob(\Gs)$.
Remark that we have $f_{\bullet}(0)=0$ for any surjective map $f\in\Gs(X,X\ppr)$. Especially if $X\ppr=G/G$, then $(\pt_X)_{\bullet}(0)=0$. Thus by condition {\rm (iii)} in Definition \ref{DefIdeal}, we have $1=(\pt_X)_{\bullet}(1)\in\I(G/G)$.
Consequently for any $Y\in\Ob(\Gs)$, we have $1=\pt_Y^{\ast}(1)\in\I(Y)$.
\end{proof}

We have the following type of ideals. (We remark that, not all the ideals are of this form. See Example \ref{ExNot}.)
\begin{prop}\label{PropIM}
Let $T$ be a Tambara functor, and $I\subseteq T(G/e)$ be a $G$-invariant ideal of $T(G/e)$. Then there exists an ideal $\I_I\subseteq T$ satisfying $\I_I(G/e)=I$.

Moreover by construction, $\I_I$ is the maximum one among the ideals $\I$ satisfying $\I(G/e)=I$.
\end{prop}
\begin{proof}
For each $X\in\Ob(\Gs)$, define $\I_I(X)$ by
\begin{equation}
\label{EqIM}
\I_I(X)=\underset{\gamma\in\Gs(G/e,X)}{\bigcap}(\gamma^{\ast})^{-1}(I).
\end{equation}
Obviously $\I_I$ satisfies $\I_I(G/e)=I$. It remains to show $\I_I$ is closed under $f^{\ast}, f_+,f_!$ for each $f\in\Gs(X,Y)$.

Let $f\in\Gs(X,Y)$ be any $G$-map. Obviously we have $f^{\ast}(\I_I(Y))\subseteq\I_I(X)$.
We show $f_+(\I_I(X))\subseteq\I(Y)$. (\, $f_!(\I_I(X))\subseteq\I(Y)$ is shown in the same way.) For any $\gamma_Y\in\Gs(G/e,Y)$, take a pull-back diagram
\[
\xy
(-6,6)*+{X\ppr}="0";
(6,6)*+{G/e}="2";
(-6,-6)*+{X}="4";
(6,-6)*+{Y}="6";
(9,-7)*+{.}="7";
(0,0)*+{\square}="8";
{\ar^{f\ppr} "0";"2"};
{\ar_{\xi} "0";"4"};
{\ar^{\gamma_Y} "2";"6"};
{\ar_{f} "4";"6"};
\endxy
\]
Then we have $\gamma_Y^{\ast}f_+(x)=f\ppr_+\xi^{\ast}(x)$ for each $x\in \I_I(X)$.
Since any point in $X\ppr$ has a stabilizer equal to $e$, we may assume this diagram is of the form
\[
\xy
(-7,6)*+{\underset{n}{\amalg} G/e}="0";
(7,6)*+{G/e}="2";
(-7,-6)*+{X}="4";
(7,-6)*+{Y}="6";
(0,0)*+{\square}="9";
{\ar "0";"2"};
{\ar "0";"4"};
{\ar^{\gamma_Y} "2";"6"};
{\ar_{f} "4";"6"};
\endxy
\]
where $\underset{n}{\amalg} G/e$ is the direct sum of $n$-copies of $G/e$ for some natural number $n$.
Thus we have
\[ \gamma_Y^{\ast}f_+(x)=\sum_{1\le i\le n}g_i\cdot \gamma_i^{\ast}(x) \]
for some $g_i\in G$ and $\gamma_i\in\Gs(G/e,X)$\ \  $(1\le i\le n)$. Since $I$ is closed under the action of $G$, this implies $\gamma_Y^{\ast}f_+(x)\in I$, and thus $f_+(x)\in\I_I(Y)$.
\end{proof}

\section{Operations on ideals}
\label{section3}
Some operations for the ideals of Tambara functors can be performed, in the same manner as in the ordinary ideal theory for commutative rings.

\begin{notation}
Let $T$ be a Tambara functor on $G$. 
If $\G$ is a subset of $\underset{X\in\Ob(\Gs)}{\coprod}T(X)$, we say \lq $\G$ is a subset of $T$', and write abbreviately $\G\subseteq T$.

For any subset $\G\subseteq T$, we put $\G(X)=\G\cap T(X)\ \ ({}^{\forall}X\in\Ob(\Gs))$. In this notation, remark that we have $\,\G=\!\!\underset{X\in\Ob(\Gs)}{\coprod}\G(X)$.
\end{notation}

\subsection{Intersection of ideals}
We can take the intersection of ideals. This enables us to consider a {\it generator} of an ideal.
\begin{prop}\label{PropIntersection}
Let $T$ be a Tambara functor and let $\{\I_{\lambda}\}_{\lambda\in\Lambda}$ be a family of ideals of $T$.
If we define a subset $\mathscr{K}\subseteq T$ by

\[ \mathscr{K}(X)=\{\underset{\lambda}{\bigcap}\,\I_{\lambda}(X)\}_{X\in\Ob(\Gs)} \]
for each $X\in\Ob(\Gs)$, then $\mathscr{K}$ becomes an ideal of $T$.
This is called the {\it intersection} of $\{\I_{\lambda}\}_{\lambda\in\Lambda}$, and denoted by $\mathscr{K}=\underset{\lambda}{\bigcap}\,\I_{\lambda}$. If $\{\I_{\lambda}\}_{\lambda\in\Lambda}$ is a finite set of ideals $\{ \I_1,\ldots, \I_{\ell}\}$, we also denote it by
\[ \underset{\lambda\in\Lambda}{\bigcap}\I_{\lambda}=\underset{1\le i\le \ell}{\bigcap}\I_i=\I_1\cap\cdots\cap\I_{\ell}. \]
\end{prop}
\begin{proof}
This immediately follows from the definition of ideals.
\end{proof}

\begin{dfn}\label{DefGen}
Let $\G\subseteq T$ be a subset. The {\it ideal generated by} $\G$ is defined as the intersection of ideals containing all elements in $\G$:
\[ \bigcap\,\{\I\subseteq T\ \text{ideal}\ \mid \G\subseteq\I\ 
 \}. \]
Obviously this is the smallest ideal containing $\G$. We denote this ideal by $\langle \G\rangle$, and call $\mathscr{G}$ the {\it generator} of this ideal.
If $\G$ is a finite set $\G=\{ a_1,a_2,\ldots,a_{\ell} \}\subseteq T$, then we denote $\langle \G\rangle$ by $\langle a_1,a_2,\ldots,a_{\ell}\rangle$. In particular if $a$ is an element of $T$, then $\langle \{ a\}\rangle$ is denoted by $\langle a\rangle$.
\end{dfn}

\begin{prop}\label{PropPrincipal}
Let $T$ be a Tambara functor, and let $a\in T(A)$ be an element for some $A\in\Ob(\Gs)$. Then $\langle a\rangle$ can be calculated as
\begin{equation}\label{EqStarStar}
\langle a\rangle(X)=\{ \alpha=u_+(c\cdot(v_!w^{\ast}(a)))\mid X\overset{u}{\leftarrow}C\overset{v}{\leftarrow}D\overset{w}{\rightarrow}A,\ c\in T(C) \}
\end{equation}
for any $X\in\Ob(\Gs)$.
\end{prop}
\begin{proof}
Denote the right hand side by $\I_a(X)$. Since $\langle a\rangle$ is closed under $u_+$, $v_!$, $w^{\ast}$, obviously we have $\langle a\rangle(X)\supseteq\I_a(X)$ for any $X\in\Ob(\Gs)$. Thus it suffices to show that $\I_a=\{ \I_a(X)\}_{X\in\Ob(\Gs)}$ forms an ideal of $T$.

First we show $\I_a(X)$ is an ideal of $T(X)$ for each $X\in\Ob(\Gs)$. Suppose we are given two elements in $\I_a(X)$
\[ \alpha=u_+(c\cdot(v_!w^{\ast}(a)))\quad \text{and} \quad \alpha\ppr=u\ppr_+(c\ppr\cdot(v\ppr_!w^{\prime\ast}(a))) \]
with
\begin{eqnarray*}
X\overset{u}{\leftarrow}C\overset{v}{\leftarrow}D\overset{w}{\rightarrow}A,\ \ \ c\in T(C),\\
X\overset{u\ppr}{\leftarrow}C\ppr\overset{v\ppr}{\leftarrow}D\ppr\overset{w\ppr}{\rightarrow}A,\ \ \ c\ppr\in T(C\ppr).
\end{eqnarray*}
Then for
\[ X\overset{u\cup u\ppr}{\longleftarrow}C\amalg C\ppr\overset{v\amalg v\ppr}{\longleftarrow}D\amalg D\ppr\overset{w\cup w\ppr}{\longrightarrow}A,\ \ \ (c,c\ppr)\in T(C\amalg C\ppr), \]
we have
\begin{eqnarray*}
\alpha+\alpha\ppr%
=(u\cup u\ppr)_+((c,c\ppr)\cdot(v\amalg v\ppr)_!(w\cup w\ppr)^{\ast}(a))\in \I_a(X).
\end{eqnarray*}
(Here we used the identification $T(C\amalg C\ppr)\cong T(C)\times T(C\ppr)$.)
Besides, for any $r\in T(X)$, we have
\[ r\alpha=r\cdot(u_+(c\cdot(v_!w^{\ast}(a))))=u_+(\,(u^{\ast}(r)\cdot c)\cdot(v_!w^{\ast}(a))\,)\in \I_a(X). \]
Thus $\I_a(X)\subseteq T(X)$ is an ideal.

It remains to show $\I_a$ is closed under $(\,)_+$, $(\,)_{\bullet}$ and $(\,)^{\ast}$.

Let $f\in\Gs(X,Y)$ be any morphism. For any $\alpha=u_+(c\cdot(v_!w^{\ast}(a)))\in\I_a(X)$ as in $(\ref{EqStarStar})$, we have
\[ f_+(\alpha)=(f\circ u)_+(c\cdot(v_!w^{\ast}(a)))\in\I_a(Y). \]
Thus $f_+(\I_a(X))\subseteq\I_a(Y)$.

If we take an exponential diagram
\[
\xy
(-14,6)*+{X}="0";
(-14,-6)*+{Y}="2";
(0,6)*+{C}="4";
(14,6)*+{Z}="6";
(14,-6)*+{C\ppr}="8";
(0,0)*+{\mathit{exp}}="10";
{\ar_{f} "0";"2"};
{\ar_{u} "4";"0"};
{\ar_>>>>>{p} "6";"4"};
{\ar^>>>>{f\ppr} "6";"8"};
{\ar^{u\ppr} "8";"2"};
\endxy
\]
and a pull-back
\[
\xy
(-6,6)*+{D}="0";
(6,6)*+{D\ppr}="2";
(-6,-6)*+{C}="4";
(6,-6)*+{Z}="6";
(9,-7)*+{,}="7";
(0,0)*+{\square}="8";
{\ar_{p\ppr} "2";"0"};
{\ar_{v} "0";"4"};
{\ar^{v\ppr} "2";"6"};
{\ar^{p} "6";"4"};
\endxy
\]
then we have
\begin{eqnarray*}
f_!(\alpha)&=&f_!u_+(c\cdot(v_!w^{\ast}(a)))\ =\ u\ppr_+f\ppr_!p^{\ast}(c\cdot(v_!w^{\ast}(a)))\\
&=&u\ppr_+(\,(f\ppr_!p^{\ast}(c))\cdot((f\ppr\circ v\ppr)_!(w\circ p\ppr)^{\ast}(a))\,)\ \in\ \I_a(Y).
\end{eqnarray*}
Thus $f_!(\I_a(X))\subseteq\I_a(Y)$.

Let $g\in\Gs(W,X)$ be any morphism. For any $\alpha=u_+(c\cdot(v_!w^{\ast}(a)))\in\I_a(X)$ as in $(\ref{EqStarStar})$, if we take pull-backs
\[
\xy
(-12,6)*+{W}="2";
(0,6)*+{C\ppr}="4";
(12,6)*+{D\ppr}="6";
(-12,-6)*+{X}="12";
(0,-6)*+{C}="14";
(12,-6)*+{D}="16";
(15,-7)*+{,}="17";
(-6,0)*+{\square}="1";
(6,0)*+{\square}="3";
{\ar_{u\ppr} "4";"2"};
{\ar_{v\ppr} "6";"4"};
{\ar^{u} "14";"12"};
{\ar^{v} "16";"14"};
{\ar_{g} "2";"12"};
{\ar^{g\ppr} "4";"14"};
{\ar^{g^{\prime\prime}} "6";"16"};
\endxy
\]
then we have
\begin{eqnarray*}
g^{\ast}(\alpha)&=&g^{\ast}u_+(c\cdot(v_!w^{\ast}(a)))\ =\ u\ppr_+g^{\prime\ast}(c\cdot(v_!w^{\ast}(a)))\\
&=&u\ppr_+(\, g^{\prime\ast}(c)\cdot(v\ppr_!(w\circ g^{\prime\prime})^{\ast}(a))\,)\ \ \in\I_a(W).
\end{eqnarray*}
Thus $g^{\ast}(\I_a(X))\subseteq\I_a(W)$.
\end{proof}

\smallskip

\subsection{Some remarks on Noetherian property}

\begin{prop}
Let $T$ be a Tambara functor on $G$. If $T(X)$ is a Noetherian ring for each transitive $X\in\Ob(\Gs)$, then any sequence of ideals
\begin{equation}
\I_1\subseteq \I_2\subseteq \I_3\subseteq\cdots\subseteq T
\label{SeqIdeal}
\end{equation}
should be stable.
\end{prop}
\begin{proof}
Assume we are given a sequence of ideals $(\ref{SeqIdeal})$.
For each transitive $X$, a sequence of ideals
\[ \I_1(X)\subseteq \I_2(X)\subseteq \I_3(X)\subseteq\cdots\subseteq T(X) \]
is induced, and there exists a natural number $n_X$ such that $\I_{n_X}(X)=\underset{n}{\cup} \I_n(X)$, since $T(X)$ is Noetherian.
If we put $n_0=\mathrm{max}\{ n_X\mid X\in\Ob(\Gs)\ \text{transitive}\}$, then we obtain $\I_{n_0}=\underset{n}{\cup}\I_n$.
\end{proof}

\begin{prop}
Let $T$ be a Tambara functor, and $\I\subseteq T$ be an ideal. If $\I$ is generated by a finite subset $\G\subseteq T$, then $\I$ can be generated by a single element.
\end{prop}
\begin{proof}
Suppose $\I$ is generated by $\G=\{ s_1,\ldots,s_n\}$ for some $s_i\in T(X_i)$.
Put $X=X_1\amalg\cdots\amalg X_n$, let $\iota_i\colon X_i\hookrightarrow X$ be the inclusion $(1\le {}^{\forall}i\le n)$, and put
\begin{equation}
s=\iota_{1+}(s_1)+\cdots+\iota_{n+}(s_n).
\label{Eq_s}
\end{equation} It suffices to show that an ideal $\J$ of $T$ contains $\G$ if and only if $\J$ contains $s$.
This follows immediately from $(\ref{Eq_s})$ and $s_i=\iota_i^{\ast}(s)\ (1\le{}^{\forall}i\le n)$.
\end{proof}

\smallskip

\subsection{Product of ideals}
\begin{prop}\label{PropProduct}
Let $\I,\J\subseteq T$ be ideals. If we define a subset $\mathscr{K}\subseteq T$ by
\[ \mathscr{K}(X)=\{ x\in T(X)\mid x=p_+(ab)\ \text{for some}\ p\in\Gs(A,X),a\in\I(A),b\in\J(A) \} \]
for each $X\in\Ob(\Gs)$, then $\mathscr{K}$ is an ideal of $T$. This is called the product of $\I$ and $\J$, and denoted by $\mathscr{K}=\I\!\!\!\J$.
\end{prop}
\begin{proof}
First we show $\mathscr{K}(X)\subseteq T(X)$ is an ideal for each $X$. Let $x=p_+(ab)$ and $x\ppr=p\ppr_+(a\ppr b\ppr)$ be any pair of elements in $\mathscr{K}(X)$, where
\begin{eqnarray*}
p\in\Gs(A,X), a\in\I(A), b\in\J(A),\\
p\ppr\in\Gs(A\ppr,X), a\ppr\in\I(A\ppr), b\ppr\in\J(A\ppr).
\end{eqnarray*}
Then, for $p\cup p\ppr\in\Gs(A\amalg A\ppr,X)$, $(a,a\ppr)\in\I(A\amalg A\ppr)$ and $(b,b\ppr)\in\J(A\amalg A\ppr)$, we have
\begin{eqnarray*}
x+x\ppr&=&p_+(ab)+p\ppr_+(a\ppr b\ppr)\ =\ (p\cup p\ppr)_+(ab,a\ppr b\ppr)\\
&=&(p\cup p\ppr)_+((a,a\ppr)\cdot(b,b\ppr))\ \in\mathscr{K}(X).
\end{eqnarray*}
(Here we used the identification $\I(A\amalg A\ppr)\cong\I(A)\times\I(A\ppr)$, $\J(A\amalg A\ppr)\cong\I(A)\times\J(A\ppr)$.)
Also for any $r\in T(X)$, we have
\[ rx=rp_+(ab)=p_+(p^{\ast}(r)ab)\ \in\mathscr{K}(X). \]
Thus $\mathscr{K}(X)$ is an ideal of $T(X)$.

To show $\mathscr{K}$ is an ideal of $T$, it suffices to show $\mathscr{K}$ is closed under $f^{\ast},f_+,f_{!}$ for any $f\in\Gs(X,Y)$. Let $x=p_+(ab)$ be any element in $\mathscr{K}(X)$ as above. Then we have $f_+(x)=(f\circ p)_+(ab)\in\K(Y)$. Thus $f_+(\K(X))\subseteq\K(Y)$. Also, if we take an exponential diagram
\[
\xy
(-14,6)*+{X}="0";
(-14,-6)*+{Y}="2";
(0,6)*+{A}="4";
(14,6)*+{Z}="6";
(14,-6)*+{\Pi}="8";
(17,-7)*+{,}="9";
(0,0)*+{\mathit{exp}}="10";
{\ar_{f} "0";"2"};
{\ar_{p} "4";"0"};
{\ar_>>>>>{\lambda} "6";"4"};
{\ar^>>>>{\rho} "6";"8"};
{\ar^{\pi} "8";"2"};
\endxy
\]
then we have
\[ f_!(x)=f_!\, p_+(ab)\ =\ \pi_+\rho_!\,\lambda^{\ast}(ab)=\pi_+(\rho_!\,\lambda^{\ast}(a)\cdot\rho_!\,\lambda^{\ast}(b))\in\K(Y). \]
Thus $f_!(\K(X))\subseteq\K(Y)$.
%

Let $y=q_+(cd)$ be any element in $\K(Y)$, where $q\in\Gs(B,Y)$, $c\in\I(B)$, $d\in\J(B)$. If we take a pull-back diagram
\[
\xy
(-6,6)*+{A}="0";
(6,6)*+{B}="2";
(-6,-6)*+{X}="4";
(6,-6)*+{Y}="6";
(9,-7)*+{,}="7";
(0,0)*+{\square}="8";
{\ar^{g} "0";"2"};
{\ar_{p} "0";"4"};
{\ar^{q} "2";"6"};
{\ar_{f} "4";"6"};
\endxy
\]
then we have
\[ f^{\ast}(y)=f^{\ast}q_+(cd)=p_+g^{\ast}(cd)=p_+(g^{\ast}(c)g^{\ast}(d))\in\K(X). \]
Thus $f^{\ast}(\K(Y))\subseteq\K(X)$.
\end{proof}

\begin{rem}\label{RemProduct}
Let $T$ be a Tambara functor.
\begin{enumerate}
\item For each ideals $\I$ and $\J$ of $T$, we have $\I\!\!\!\J\subseteq\I\cap\J$.
\item For each ideals $\I$ and $\J$ of $T$, we have
\[ \I(X)\!\J(X)\subseteq\I\!\!\!\J(X) \]
for any $X\in\Ob(\Gs)$. Here, the left hand side is the ordinary product of ideals in $T(X)$.
\item If $\I_1\subseteq \J_1$, $\I_2\subseteq \J_2$ are ideals of $T$, then we have $\I_1\!\I_2\subseteq\J_1\!\J_2$.
\end{enumerate}
\end{rem}
\begin{proof}
These immediately follows from Proposition \ref{PropIntersection}.
\end{proof}

\begin{prop}\label{PropFeb02}
If $\I_1,\I_2,\I_3$ are ideals of a Tambara functor $T$, then we have
\[ (\I_1\I_2)\I_3=\I_1(\I_2\I_3). \]
More precisely, for any $X\in\Ob(\Gs)$, both $((\I_1\I_2)\I_3)(X)$ and $(\I_1(\I_2\I_3))(X)$ agree with the set
\[ \K_{1,2,3}(X)=\Set{ p_+(a_1a_2a_3)\in T(X)|\, p\in\Gs(A,X),\ \begin{array}{c}a_i\in\I_i(A)\\ (i=1,2,3)\end{array} }. \]
Inductively, for any ideals $\I_1,\I_2,\ldots,\I_{\ell}\subseteq T$, we have
\[ (\I_1\I_2\cdots\I_{\ell})(X)=\Set{ p_+(a_1a_2\cdots a_{\ell})|\, p\in\Gs(A,X),\ \begin{array}{c}a_i\in\I_i(A)\\ (1\le i\le \ell) \end{array} }. \]
Especially, $\I_1\I_2\cdots\I_{\ell}$ does not depend on the order of the products. For each $n\in\mathbb{N}$, we denote the $n$-times product of an ideal $\I\subseteq T$ by $\I^n$.
\end{prop}
\begin{proof}
We only show
\[ ((\I_1\I_2)\I_3)(X)=\K_{1,2,3}(X) \]
for each $X\in\Ob(\Gs)$. $\K_{1,2,3}(X)=(\I_1(\I_2\I_3))(X)$ can be shown in the same way.

By definition of the product, we have
\begin{equation}\label{Eq_Three}
((\I_1\I_2)\I_3)(X)=\Set{ p_+(b_1q_+(b_2b_3))| \begin{array}{c}p\in\Gs(A,X),\\ q\in\Gs(B,A),\end{array}\ \begin{array}{c}b_1\in\I_1(A),\\ b_2\in\I_2(B),\\ b_3\in\I_3(B)\end{array} }.
\end{equation}

Thus $((\I_1\I_2)\I_3)(X)\supseteq \K_{1,2,3}(X)$ follows if we put $B=A$, $q=\id_A$ and $b_i=a_i\ (1\le i\le 3)$.

Conversely, any element $p_+(b_1q_+(b_2b_3))$ as in $(\ref{Eq_Three})$ is written as
\begin{eqnarray*}
p_+(b_1q_+(b_2b_3))&=&p_+(q_+(q^{\ast}(b_1)b_2b_3))\\
&=&(p\circ q)_+(q^{\ast}(b_1)b_2b_3))
\end{eqnarray*}
by the projection formula, which means $((\I_1\I_2)\I_3)(X)\subseteq \K_{1,2,3}(X)$.
\end{proof}

\begin{lem}\label{LemPullAdd}
Let
\[
\xy
(-6,6)*+{C_1}="0";
(6,6)*+{C}="2";
(-6,-6)*+{B}="4";
(6,-6)*+{C_2}="6";
(0,0)*+{\square}="8";
{\ar_{p_1} "2";"0"};
{\ar_{u_1} "0";"4"};
{\ar^{p_2} "2";"6"};
{\ar^{u_2} "6";"4"};
\endxy
\]
be a pull-back diagram in $\Gs$, and put
\[ u=u_1\circ p_1=u_2\circ p_2. \]
Then for any $\alpha_1\in T(C_1)$ and $\alpha_2\in T(C_2)$, we have
\[ u_{1+}(\alpha_1)\cdot u_{2+}(\alpha_2)=u_+(p_1^{\ast}(\alpha_1)\cdot p_2^{\ast}(\alpha_2)). \]
\end{lem}
\begin{proof}
By the projection formula, we have
\begin{eqnarray*}
u_{1+}(\alpha_1)\cdot u_{2+}(\alpha_2)&=&u_{2+}((u_2^{\ast}u_{1+}(\alpha_1))\cdot\alpha_2)\ =\ u_{2+}((p_{2+}p_1^{\ast}(\alpha_1))\cdot\alpha_2)\\
&=&u_{2+}(p_{2+}(p_1^{\ast}(\alpha_1)\cdot p_2^{\ast}(\alpha_2)))\ =\ u_+(p_1^{\ast}(\alpha_1)\cdot p_2^{\ast}(\alpha_2)).
\end{eqnarray*}
\end{proof}

\begin{prop}\label{PropTwoProd}
Let $T$ be a Tambara functor, and let $a_1\in T(A_1)$, $a_2\in T(A_2)$ be elements for some $A_1,A_2\in\Ob(\Gs)$. Then, $\langle a_1\rangle\langle a_2\rangle$ can be calculated as
\[%
(\langle a_1\rangle\langle a_2\rangle)(X)=%
\Set{ u_+(c\,\delta_1\delta_2) |\xy
(-12,0)*+{X}="0";
(-2,0)*+{C}="1";
(9,4)*+{D_1}="2";
(20,4)*+{A_1}="4";
(9,-4)*+{D_2}="12";
(20,-4)*+{A_2}="14";
{\ar_{u} "1";"0"};
{\ar_{v_1} "2";"1"};
{\ar^{w_1} "2";"4"};
{\ar^{v_2} "12";"1"};
{\ar^{w_2} "12";"14"};
\endxy
\quad%
\begin{array}{c}%
\delta_1=v_{1!}w_1^{\ast}(a_1)\\%
\delta_2=v_{2!}w_2^{\ast}(a_2)\\%
c\in T(C)
\end{array}%
} \]
for any $X\in\Ob(\Gs)$.
\end{prop}
\begin{proof}
Denote the right hand side by $\I_{a_1,a_2}(X)$. By Proposition \ref{PropPrincipal} and Proposition \ref{PropProduct}, $(\langle a_1\rangle\langle a_2\rangle)(X)$ is equal to the set
\begin{equation}
\Set{ p_+(u_{1+}(\alpha_1)u_{2+}(\alpha_2)) |%
\begin{array}{c}
B\overset{u_1}{\leftarrow}C_1\overset{v_1}{\leftarrow}D_1\overset{w_1}{\rightarrow}A_1\\
B\overset{u_2}{\leftarrow}C_2\overset{v_2}{\leftarrow}D_2\overset{w_2}{\rightarrow}A_2\\
\end{array}
\begin{array}{c}%
\alpha_1=c_1\cdot(v_{1!}w_1^{\ast}(a_1))\\%
\alpha_2=c_2\cdot(v_{2!}w_2^{\ast}(a_2))\\%
c_1\in T(C_1),\ c_2\in T(C_2)\\%
p\in\Gs(B,X)
\label{EqStarStarStar}
\end{array}%
}.
\end{equation}
Thus $(\langle a_1\rangle\langle a_2\rangle)(X)\supseteq\I_{a_1,a_2}(X)$ is shown, by letting $B=C_1=C_2=C,\ u_1=u_2=\id_C,\ c_1=c,\ c_2=1$ and $p=u$ in $(\ref{EqStarStarStar})$.

Conversely, we show $(\langle a_1\rangle\langle a_2\rangle)(X)\subseteq\I_{a_1,a_2}(X)$. Let $\omega=p_+(u_{1+}(\alpha_1)u_{2+}(\alpha_2))$ be any element in $(\langle a_1\rangle\langle a_2\rangle)(X)$ as in $(\ref{EqStarStarStar})$.
If we take pull-backs
\[
\xy
(-12,12)*+{D_1}="0";
(0,12)*+{D_1\ppr}="2";
(-12,0)*+{C_1}="4";
(0,0)*+{C}="6";
(12,0)*+{D_2\ppr}="8";
(-12,-12)*+{B}="10";
(0,-12)*+{C_2}="12";
(12,-12)*+{D_2}="14";
(15,-13)*+{,}="15";
(-6,6)*+{\square}="21";
(-6,-6)*+{\square}="23";
(6,-6)*+{\square}="25";
{\ar_{q_1} "2";"0"};
{\ar_{p_1} "6";"4"};
{\ar_{u_2\ppr} "8";"6"};
{\ar^{u_2} "12";"10"};
{\ar^{v_2} "14";"12"};
{\ar_{v_1} "0";"4"};
{\ar^{v_1\ppr} "2";"6"};
{\ar_{u_1} "4";"10"};
{\ar^{p_2} "6";"12"};
{\ar^{q_2} "8";"14"};
\endxy
\]
and put $u=u_1\circ p_1=u_2\circ p_2$, then by Lemma \ref{LemPullAdd}, we have
\begin{eqnarray*}
\omega&=&p_+(u_{1+}(\alpha_1)u_{2+}(\alpha_2))\ =\ p_+u_+(p_1^{\ast}(\alpha_1)p_2^{\ast}(\alpha_2))\\
&=&(p\circ u)_+(\, p_1^{\ast}(c_1)\cdot p_2^{\ast}(c_2)\cdot(v\ppr_{1!}(w_1\circ q_1)^{\ast}(a_1))\cdot(v\ppr_{2!}(w_2\circ q_2)^{\ast}(a_2))\,)\\
&\in& \I_{a_1,a_2}(X),
\end{eqnarray*}
and thus $(\langle a_1\rangle\langle a_2\rangle)(X)\subseteq\I_{a_1,a_2}(X)$.
\end{proof}

\begin{cor}\label{CorTwoProd}
Let $\I\subseteq T$ be an ideal, and let $a\in T(A)$, $b\in T(B)$ be elements for some $A,B\in\Ob(\Gs)$.
Then, the following are equivalent.
\begin{enumerate}
\item $\langle a\rangle\langle b\rangle\subseteq\I$.
\item For any $C\in\Ob(\Gs)$ and any pair of diagrams
\begin{equation}
\label{EqStarStarStarStar}
C\overset{v}{\leftarrow}D\overset{w}{\rightarrow}A,\qquad C\overset{v\ppr}{\leftarrow}D\ppr\overset{w\ppr}{\rightarrow}B,
\end{equation}
$\I(C)$ satisfies $(v_!w^{\ast}(a))\cdot(v\ppr_!w^{\prime\ast}(b))\in\I(C)$.
\end{enumerate}
\end{cor}
\begin{proof}
By Proposition \ref{PropTwoProd}, $\langle a\rangle\langle b\rangle\subseteq\I$ if and only if for any
\[ X,C\in\Ob(\Gs),\ \ p\in\Gs(C,X),\ \ c\in T(C), \]
and any pair of diagrams $(\ref{EqStarStarStarStar})$,
\[ p_+(c\cdot(v_!w^{\ast}(a))\cdot(v\ppr_!w^{\prime\ast}(b)))\in\I(X) \]
is satisfied.
Since $\I$ is an ideal, if $(v_!w^{\ast}(a))\cdot(v\ppr_!w^{\prime\ast}(b))\in\I(C)$, then automatically $p_+(c\cdot(v_!w^{\ast}(a))\cdot(v\ppr_!w^{\prime\ast}(b)))\in\I(X)$ follows.
Thus $\langle a\rangle\langle b\rangle\subseteq\I$ if and only if $\I$ satisfies condition {\rm (2)}.
\end{proof}

\smallskip

\subsection{Sum of ideals}

\begin{prop}\label{PropSum}
Let $\{ \I_{\lambda}\}_{\lambda\in\Lambda}$ be a family of ideals in $T$. If we define a subset $\mathscr{K}\subseteq T$ by
\[ \mathscr{K}(X)=\sum_{\lambda\in\Lambda}\I_{\lambda}(X) \]
for each $X\in\Ob(\Gs)$, where the right hand side is the ordinary sum of ideals in $T(X)$, then $\mathscr{K}$ is an ideal of $T$. This is called the {\it sum} of $\{\I_{\lambda}\}_{\lambda\in\Lambda}$, and denoted by $\mathscr{K}=\underset{\lambda\in\Lambda}{\sum}\I_{\lambda}$. Obviously, this is the smallest ideal containing $\I_{\lambda}$ for all $\lambda$.
If $\{\I_{\lambda}\}_{\lambda\in\Lambda}$ is a finite set of ideals $\{ \I_1,\ldots, \I_{\ell}\}$, we also denote it by
\[ \sum_{\lambda\in\Lambda}\I_{\lambda}=\sum_{1\le i\le \ell}\I_i=\I_1+\cdots+\I_{\ell}. \]
\end{prop}
\begin{proof}
Since the other conditions are trivial, we only show $\K$ satisfies
\[ f_!(\K(X))\subseteq\K(Y) \]
for any $f\in\Gs(X,Y)$.
Namely, for any finite sum $x=\underset{1\le i\le k}{\sum}a_i$ with $a_i\in\I_{\lambda_i}(X)$ and for any $f$, we show $f_!(x)\in\K(Y)$.
Obviously this is reduced to the following case of the finite sum:
\begin{claim}\label{ClaimFinSum}
If $\I_1,\ldots,\I_{\ell}$ are ideals of $T$, then a subset $\K\subseteq T$ defined by
\[ \K(X)=\I_1(X)+\cdots+\I_{\ell}(X)\ \ ({}^{\forall}X\in\Ob(\Gs)) \]
becomes an ideal $\K\subseteq T$. 
\end{claim}

\smallskip

Moreover, by an induction, it suffices to show Claim \ref{ClaimFinSum} for $\ell=2$.

\medskip

Let $\I,\J$ be ideals, and let $a\in\I(X)$, $b\in\J(X)$ be any pair of elements. Then for any $f\in\Gs(X, Y)$, we have
\begin{equation}\label{EqTwoSum}
f_{\bullet}(a+b)=s_+(t_{\bullet}r^{\ast}(a)\cdot t\ppr_{\bullet}r^{\prime\ast}(b))
\end{equation}
in the notation of Remark \ref{RemExp}.
Since $t_{\bullet}(\I(U))\subseteq t_{\bullet}(0)+\I(V)$, there exists an element $\alpha\in\I(V)$ such that $t_{\bullet}r^{\ast}(a)=t_{\bullet}(0)+\alpha$.
Similarly there is $\beta\in\J(V)$ such that $t\ppr_{\bullet}r^{\prime\ast}(b)=t\ppr_{\bullet}(0)+\beta$.
Then we have
\begin{eqnarray*}
(\ref{EqTwoSum})&=&s_+(t_{\bullet}(0)t\ppr_{\bullet}(0)+\alpha t\ppr_{\bullet}(0)+t_{\bullet}(0)\beta+\alpha\beta)\\
&\subseteq&s_+(t_{\bullet}(0)t\ppr_{\bullet}(0))+\I(Y)+\J(Y)+(\I\cap\J)(Y)\\
&=&f_{\bullet}(0)+\I(Y)+\J(Y)+(\I\cap\J)(Y)\\
&\subseteq& f_{\bullet}(0)+\I(Y)+\J(Y).
\end{eqnarray*}
Thus Claim \ref{ClaimFinSum} is shown, and Proposition \ref{PropSum} follows.
\end{proof}

\begin{prop}\label{PropSumProd}
If $\I,\J,\I\ppr,\J\ppr\subseteq T$ are ideals, we have
\[ (\I+\J)(\I\ppr+\J\ppr)=\I\!\!\I\ppr+\I\!\!\J\ppr+\I\ppr\!\!\J+\J\!\!\J\ppr. \]
\end{prop}
\begin{proof}
First we show $(\I+\J)(\I\ppr+\J\ppr)\subseteq\I\!\!\I\ppr+\I\!\!\J\ppr+\I\ppr\!\!\J+\J\!\!\J\ppr$.
For any $X\in\Ob(\Gs)$, any element $\alpha\in(\I+\J)(\I\ppr+\J\ppr)(X)$ is of the form
\[ \alpha=p_+((a+b)(a\ppr+b\ppr)), \]
where $p\in\Gs(A,X),\ a\in\I(A),\ a\ppr\in\I\ppr(A),\ b\in\J(A),\ b\ppr\in\J\ppr(A)$, and thus we have
\begin{eqnarray*}
\alpha&=&p_+(aa\ppr)+p_+(ab\ppr)+p_+(ba\ppr)+p_+(bb\ppr)\\
&\in&(\I\!\!\I\ppr+\I\!\!\J\ppr+\I\ppr\!\!\J+\J\!\!\J\ppr)(X).
\end{eqnarray*}

Conversely, let $\beta$ be any element in $(\I\!\!\I\ppr+\I\!\!\J\ppr+\I\ppr\!\!\J+\J\!\!\J\ppr)(X)$. There exist
\begin{eqnarray*}
&p_1\in\Gs(A_1,X),\ a_1\in\I(A_1),\ a\ppr_1\in\I\ppr(A_1)&\\
&p_2\in\Gs(A_2,X),\ a_2\in\I(A_2),\ b\ppr_2\in\J\ppr(A_2)&\\
&p_3\in\Gs(A_3,X),\ a\ppr_3\in\I\ppr(A_3),\ b_3\in\J(A_3)&\\
&p_4\in\Gs(A_4,X),\ b_4\in\J(A_4),\ b\ppr_4\in\J\ppr(A_4)&
\end{eqnarray*}
such that
\begin{equation}\label{Eq_beta}
\beta=p_{1+}(a_1a\ppr_1)+p_{2+}(a_2b\ppr_2)+p_{3+}(a\ppr_3b_3)+p_{4+}(b_4b\ppr_4).
\end{equation}
If we put $A=A_1\amalg A_2\amalg A_3\amalg A_4$ and put
\begin{eqnarray*}
&a=(a_1,a_2,0,0)\ \ \in\I(A)&\\
&a\ppr=(a\ppr_1,0,a\ppr_3,0)\ \ \in\I\ppr(A)&\\
&b=(0,0,b_3,b_4)\ \ \in\J(A)&\\
&b\ppr=(0,b\ppr_2,0,b\ppr_4)\ \ \in\J\ppr(A)&
\end{eqnarray*}
(here, we used the identification $\I(A)\cong\I(A_1)\times\I(A_2)\times\I(A_3)\times\I(A_4)$, and so on), then we have
\begin{eqnarray*}
(\I+\J)(\I\ppr+\J\ppr)(X)&\ni&p_+((a+b)(a\ppr+b\ppr))\\
&=&p_+(a_1a\ppr_1,a_2b\ppr_2,b_3a\ppr_3,b_4b\ppr_4)\\
&=&p_{1+}(a_1a\ppr_1)+p_{2+}(a_2b\ppr_2)+p_{3+}(b_3a\ppr_3)+p_{4+}(b_4b\ppr_4).
\end{eqnarray*}
By $(\ref{Eq_beta})$, we obtain $\beta\in(\I+\J)(\I\ppr+\J\ppr)(X)$.
\end{proof}

\begin{dfn}\label{DefCoprime}
Let $\I,\J\subseteq T$ be ideals. $\I$ and $\J$ are {\it coprime} if $\I+\J=T$. If $\I_1,\ldots,\I_{\ell}\subseteq T$ are ideals, $\I_1,\ldots,\I_{\ell}$ are {\it pairwise coprime} if $\I_i$ and $\I_j$ are coprime for any $1\le i< j\le\ell$.
\end{dfn}

\begin{rem}\label{RemCoprime}
Let $\I,\J\subseteq T$ be ideals. By Proposition \ref{PropTrivIdeal} and Proposition \ref{PropSum}, the following are equivalent.
\begin{enumerate}
\item $\I$ and $\J$ are coprime.
\item $\I(X)$ and $\J(X)$ are coprime for any $X\in\Ob(\Gs)$.
\item $\I(X)$ and $\J(X)$ are coprime for some $\emptyset\ne X\in\Ob(\Gs)$.
\end{enumerate}
\end{rem}

\begin{cor}\label{CorCoprime}
Let $\I,\J\subseteq T$ be ideals. If $\I$ and $\J$ are coprime, then we have
\[ \I(X)\J(X)=\I\!\!\J(X) \]
for any $X\in\Ob(\Gs)$. Here, the left hand side is the ordinary product of ideals in $T(X)$.
\end{cor}
\begin{proof}
Assume $\I$ and $\J$ are coprime. By Remark \ref{RemCoprime}, $\I(X)$ and $\J(X)$ are coprime for any $X\in\Ob(\Gs)$.
Then by the ordinary ideal theory for commutative rings, we have
\[ \I(X)\J(X)=\I(X)\cap\J(X). \]
On the other hand, by Remark \ref{RemProduct}, we have
\[ \I(X)\J(X)\subseteq\I\!\!\J(X)\subseteq\I(X)\cap\J(X). \]
Thus it follows $\I(X)\J(X)=\I\!\!\J(X)$.
\end{proof}

From this, the Chinese remainder theorem for a Tambara functor immediately follows:
\begin{cor}\label{CorChinese}
If $\I_1,\ldots,\I_{\ell}\subseteq T$ are pairwise coprime ideals, then we have the following.
\begin{enumerate}
\item $(\I_1\!\cdots\!\I_{\ell})(X)=\I_1(X)\!\cdots\!\I_{\ell}(X)$ for any $X\in\Ob(\Gs)$.
\item $\I_1\!\cdots\!\I_{\ell}=\I_1\cap\cdots\cap\I_{\ell}$.
\item The projections $T/(\I_1\!\cdots\!\I_{\ell})\rightarrow T/\I_i\ \ (1\le i\le \ell)$ induce an isomorphism of Tambara functors
\[ T/(\I_1\!\cdots\!\I_{\ell})\overset{\cong}{\longrightarrow}T/\I_1\times\cdots\times T/\I_{\ell}. \]
\end{enumerate}
\end{cor}
\begin{proof}
{\rm (1)} follows immediately from Corollary \ref{CorCoprime}, by induction. 
By Corollary \ref{CorCoprime}, {\rm (2)} and {\rm (3)} can be confirmed on each $X\in\Ob(\Gs)$, which is well known in the ordinary ideal theory for commutative rings.
\end{proof}


\section{Prime ideals}
\label{section4}
In the rest of this article, we assume $T$ is non-trivial, namely $T\ne 0$.

\subsection{Definition of a prime ideal}

\begin{dfn}\label{DefPrime}
An ideal $\p\subsetneq T$ is {\it prime} if for any transitive $X,Y\in\Ob(\Gs)$ and any $a\in T(X)$, $b\in T(Y)$,
\[ \langle a\rangle\langle b\rangle\subseteq\p\ \ \Rightarrow\ \ a\in\p(X)\ \ \text{or}\ \ b\in\p(Y) \]
is satisfied. Remark that the converse always holds.

An ideal $\m\subsetneq T$ is {\it maximal} if it is maximal with respect to the inclusion of ideals not equal to $T$.
\end{dfn}

\begin{rem}\label{RemPrime}
If $\I\subsetneq T$ is an ideal not equal to $T$, then by Zorn's lemma, there exists some maximal ideal $\m\subsetneq T$ containing $\I$.

Obviously, when $G$ is trivial, the definition of a prime (resp. maximal) ideal agrees with the ordinary definition of a prime (resp. maximal) ideal.
\end{rem}

\begin{prop}\label{PropPrimeMax}
Any maximal ideal $\m\subseteq T$ is prime.
\end{prop}
\begin{proof}
Let $X,Y\in\Ob(\Gs)$ be transitive, and let $a\in T(X)$, $b\in T(Y)$. Suppose $a\in\!\!\!\!\! /\ \m(X)$ and $b\in\!\!\!\!\! /\ \m(Y)$. This means $\m\subsetneq (\m+\langle a\rangle)$ and $\m\subsetneq (\m+\langle b\rangle)$, and we have $\m+\langle a\rangle=T$ and $\m+\langle b\rangle=T$ by the maximality of $\m$. Thus $(\m+\langle a\rangle)(\m+\langle b\rangle)=TT=T$.
Since
\begin{eqnarray*}
(\m+\langle a\rangle)(\m+\langle b\rangle)&=&\m\m+\m\langle a\rangle+\langle b\rangle\m+\langle a\rangle\langle b\rangle\\
&\subseteq&\m+\langle a\rangle\langle b\rangle
\end{eqnarray*}
by Proposition \ref{PropSumProd}, it follows $\m+\langle a\rangle\langle b\rangle=T$, which means $\langle a\rangle\langle b\rangle\subseteq\!\!\!\!\!\! /\ \m$.
\end{proof}

\begin{prop}\label{PropPrimePrime}
Let $\p\subsetneq T$ be an ideal. The following are equivalent.
\begin{enumerate}
\item $\p$ is prime.
\item For any pair of ideals $\I,\J\subseteq T$,
\[ \I\!\!\!\J\subseteq\p\ \ \Rightarrow\ \ \I\subseteq\p\ \ \text{or}\ \ \J\subseteq\p \]
is satisfied. $($The converse always holds.$)$
\end{enumerate}
\end{prop}
\begin{proof}
Obviously {\rm (2)} implies {\rm (1)}. We show the converse. Suppose $\J\subseteq\!\!\!\!\!\!\! /\ \p$ and $\I\!\!\!\J\subseteq\p$. By $\J\subseteq\!\!\!\!\!\! /\ \p$, there exist some transitive $Y\in\Ob(\Gs)$ and $b\in\J(Y)$ such that $b\in\!\!\!\!\! /\ \p(Y)$. For any transitive $X\in\Ob(\Gs)$ and any $a\in\I(X)$, we have
\[ \langle a\rangle\langle b\rangle\subseteq\I\!\!\!\J\subseteq\p. \]
Since $\p$ is prime, we obtain $a\in\p(X)$. Thus $\I(X)\subseteq\p(X)$ for any transitive $X\in\Ob(\Gs)$, which means $\I\subseteq\p$.
\end{proof}

By virtue of Corollary \ref{CorTwoProd}, Definition \ref{DefPrime} is rewritten as follows.
\begin{cor}\label{CorPrimeDef}
An ideal $\p\subsetneq T$ is prime if and only if the following two conditions become equivalent, for any transitive $X,Y\in\Ob(\Gs)$ and any $a\in T(X)$, $b\in T(Y):$
\begin{enumerate}
\item For any $C\in\Ob(\Gs)$ and for any pair of diagrams
\[ C\overset{v}{\leftarrow}D\overset{w}{\rightarrow}X,\quad C\overset{v\ppr}{\leftarrow}D\ppr\overset{w\ppr}{\rightarrow}Y, \]
$\p(C)$ satisfies
\[ (v_!w^{\ast}(a))\cdot(v\ppr_!w^{\prime\ast}(b))\in\p(C). \]
\item $a\in\p(X)$ or $b\in\p(Y)$.
\end{enumerate}
Also remark that {\rm (2)} always implies {\rm (1)}.
\end{cor}

\smallskip

\subsection{Spectrum of a Tambara functor}

\begin{dfn}\label{DefSpec}
For any Tambara functor $T$ on $G$, define $\Spec(T)$ to be the set of all prime ideals of $T$. For each ideal $\I\subseteq T$, define a subset $V(\I)\subseteq\Spec(T)$ by
\[ V(\I)=\{ \p\in\Spec(T)\mid \I\subseteq\p \}. \]
\end{dfn}

\begin{rem}\label{RemSpec}
For any Tambara functor $T$, we have the following.
\begin{enumerate}
\item $V(\I)=\emptyset$ if and only if $\I=T$.
\item $V(\I)=\Spec(T)$ if and only if $\I\subseteq\underset{\p\in\Spec(T)}{\bigcap}\p$.
\end{enumerate}
\end{rem}
\begin{proof}
\begin{enumerate}
\item $V(T)=\emptyset$ is obvious. The converse follows from Remark \ref{RemPrime} and Proposition \ref{PropPrimeMax}.
\item This is trivial.
\end{enumerate}
\end{proof}

\begin{prop}\label{PropSpec}
For any Tambara functor $T$, the family 
$\{ V(\I)\mid \I\subseteq T\ \text{ideal}\}$ forms a system of closed subsets of $\Spec(T)$.
\end{prop}
\begin{proof}
By Remark \ref{RemSpec}, we have $V((0))=\Spec(T)$ and $V(T)=\emptyset$. By Proposition \ref{PropSum}, we have $\underset{\lambda\in\Lambda}{\bigcap}V(\I_{\lambda})=V(\underset{\lambda\in\Lambda}{\sum}\I_{\lambda})$ for any family of ideals $\{\I_{\lambda}\}_{\lambda\in\Lambda}$. By Proposition \ref{PropPrimePrime}, we have $V(\I\!\!\!\J)=V(\I)\cup V(\J)$ for any pair of ideals $\I,\J\subseteq T$.
\end{proof}

Regarding Proposition \ref{Prop1to1}, we have the following.
\begin{thm}\label{ThmMorphSpec}
Let $\varphi\colon T\rightarrow S$ be a morphism in $\TamG$.
\begin{enumerate}
\item $\varphi^{-1}(\q)\subseteq T$ becomes prime if $\q\subsetneq S$ is a prime ideal. This gives a continuous map
\[ \varphi^{a}\colon \Spec(S)\rightarrow\Spec(T)\ ;\ \q\mapsto\varphi^{-1}(\q). \]
\item Assume $\varphi$ is surjective. Then, $\varphi(\p)\subseteq S$ becomes prime if $\p\subsetneq T$ is a prime ideal containing $\Ker\,\varphi$. This gives a continuous map
\[ \varphi_{\sharp}\colon V(\Ker\,\varphi)\rightarrow\Spec(S)\ ;\ \p\mapsto\varphi(\p). \]
\end{enumerate}
\end{thm}
\begin{proof}
We use the criterion of Corollary \ref{CorPrimeDef}.
\begin{enumerate}
\item Let $\q\subsetneq S$ be a prime ideal. Firstly, since $1\in\varphi^{-1}(\q)(X)$ would imply $1\in\q(X)\ \ ({}^{\forall}X\in\Ob(\Gs))$, we have $\varphi^{-1}(\q)\subsetneq T$.

Let $a\in T(X)$, $b\in T(Y)$ be any pair of elements with $X,Y\in\Ob(\Gs)$ transitive, satisfying
\[ a\in\!\!\!\!\!/\ \varphi^{-1}(\q)(X)=\varphi_X^{-1}(\q(X)),\quad b\in\!\!\!\!\!/\ \varphi^{-1}(\q)(Y)=\varphi_Y^{-1}(\q(Y)). \]
Since $\q$ is prime, they satisfy
\[ (v_!w^{\ast}(\varphi_X(a)))\cdot(v\ppr_!w^{\prime\ast}(\varphi_Y(b)))\in\!\!\!\!\!/\ \q(C) \]
for some
\[ C\overset{v}{\leftarrow}D\overset{w}{\rightarrow}X,\quad C\overset{v\ppr}{\leftarrow}D\ppr\overset{w\ppr}{\rightarrow}Y. \]
Thus we have
\[ \varphi_C((v_!w^{\ast}(a))\cdot(v\ppr_!w^{\prime\ast}(b)))=(v_!w^{\ast}(\varphi_X(a)))\cdot(v\ppr_!w^{\prime\ast}(\varphi_Y(b)))\in\!\!\!\!\!/\ \q(C), \]
which means $(v_!w^{\ast}(a))\cdot(v\ppr_!w^{\prime\ast}(b))\in\!\!\!\!\!/\ \varphi^{-1}(\q)(C)$, and thus $\varphi^{-1}(\q)$ is prime.

Moreover, we have
\begin{eqnarray*}
(\varphi^a)^{-1}(V(\I))&=&\{ \q\in\Spec(S)\mid \I\subseteq\varphi^a(\q) \}\\
&=&\{ \q\in\Spec(S)\mid \langle\varphi(\I)\rangle\subseteq\q \}\\
&=&V(\langle\varphi(\I)\rangle)
\end{eqnarray*}
for each ideal $\I\subseteq T$, which means $\varphi^a$ is a continuous map.
Here, $\varphi(\I)\subseteq S$ is a subset defined by $(\varphi(\I))(X)=\varphi_X(\I(X))$.

\item Let $\p\subsetneq T$ be a prime ideal containing $\Ker\,\varphi$. Firstly, since $\Ker\,\varphi\subseteq\p$ and $1\in\varphi(\p)(X)$ would imply $1\in\p(X)\ \ ({}^{\forall}X\in\Ob(\Gs))$, we have $\varphi(\p)\subsetneq S$.

Let $c\in S(X)$, $d\in S(Y)$ be any pair of elements with $X,Y\in\Ob(\Gs)$ transitive, satisfying
\begin{equation}
\label{EqMorphSpec}
c\in\!\!\!\!\!/\ \varphi(\p)(X)=\varphi_X(\p(X)),\quad d\in\!\!\!\!\!/\ \varphi(\p)(Y)=\varphi_Y(\p(Y)).
\end{equation}
Since $\varphi$ is surjective, there exists some elements $a\in T(X)$ and $b\in T(Y)$ such that $\varphi_X(a)=c$, $\varphi_Y(b)=d$. By $(\ref{EqMorphSpec})$, it follows
\[ a\in\!\!\!\!\!/\ \p(X),\quad b\in\!\!\!\!\!/\ \p(Y). \]

Since $\p$ is prime, they satisfy
\[ (v_!w^{\ast}(a))\cdot(v\ppr_!w^{\prime\ast}(b))\in\!\!\!\!\!/\ \p(C) \]
for some
\[ C\overset{v}{\leftarrow}D\overset{w}{\rightarrow}X,\quad C\overset{v\ppr}{\leftarrow}D\ppr\overset{w\ppr}{\rightarrow}Y. \]
By $\Ker\,\varphi\subseteq\p$, we have
\[ \varphi_C((v_!w^{\ast}(a))\cdot(v\ppr_!w^{\prime\ast}(b)))\in\!\!\!\!\!/\ \varphi_C(\p(C)), \]
and thus
\begin{eqnarray*}
(v_!w^{\ast}(c))\cdot(v\ppr_!w^{\prime\ast}(d))&=&(v_!w^{\ast}\varphi_X(a))\cdot(v\ppr_!w^{\prime\ast}\varphi_Y(b))\\
&=&\varphi_C((v_!w^{\ast}(a))\cdot(v\ppr_!w^{\prime\ast}(b)))\\
&\in\!\!\!\!\!/\ &\varphi_C(\p(C)),
\end{eqnarray*}
which means $\varphi(\p)\subseteq S$ is prime.

Moreover, we have
\begin{eqnarray*}
(\varphi_{\sharp})^{-1}(V(\J))&=&\{ \p\in V(\Ker\,\varphi)\mid \J\subseteq\varphi_{\sharp}(\p) \}\\
&=&\{ \p\in V(\Ker\,\varphi)\mid \varphi^{-1}(\J)\subseteq\p \}\\
&=&V(\varphi^{-1}(\J))
\end{eqnarray*}
for each ideal $\J\subseteq S$, which means $\varphi_{\sharp}$ is a continuous map.
\end{enumerate}
\end{proof}

\begin{cor}\label{CorHomeoSpec}
Let $T$ be a Tambara functor and let $\I\subseteq T$ be an ideal. Then the projection $p\colon T\rightarrow T/\I$ induces a homeomorphism between $\Spec(T/\I)$ and $V(\I)\subseteq\Spec(T)$
\[ p^a\colon \Spec(T/\I)\overset{\approx}{\longrightarrow}V(\I), \]
with the inverse $p_{\sharp}$.
\end{cor}
\begin{proof}
This immediately follows from Theorem \ref{ThmMorphSpec}.
\end{proof}



\begin{dfn}
A Tambara functor $T$ is {\it reduced} if it satisfies
\[ \underset{\p\in\Spec(T)}{\bigcap}\p=0. \]
\end{dfn}

\begin{rem}\label{RemInt}
If $\varphi\colon T\rightarrow S$ is a surjective map of Tambara functors, then for any family of ideals $\{ \I_{\lambda}\}_{\lambda\in\Lambda}$ of $T$ containing $\Ker\,\varphi$, we have
\[ \varphi(\underset{\lambda\in\Lambda}{\bigcap}\I_{\lambda})=\underset{\lambda\in\Lambda}{\bigcap}\varphi(\I_{\lambda}). \]
\end{rem}
\begin{proof}
By Proposition \ref{Prop1to1}, there is an ideal $\J_{\lambda}\subset S$ satisfying $\I_{\lambda}=\varphi^{-1}(\J_{\lambda})$ for each $\lambda\in\Lambda$.
Since we have
\begin{eqnarray*}
\underset{\lambda\in\Lambda}{\bigcap}\J_{\lambda}%
&=&\varphi(\varphi^{-1}(\underset{\lambda\in\Lambda}{\bigcap}\J_{\lambda}))%
\,=\,\varphi(\underset{\lambda\in\Lambda}{\bigcap}\varphi^{-1}(\J_{\lambda}))\\
&\subseteq&\underset{\lambda\in\Lambda}{\bigcap}\varphi(\varphi^{-1}(\J_{\lambda}))%
\,=\,\underset{\lambda\in\Lambda}{\bigcap}\J_{\lambda},
\end{eqnarray*}
we have
\[ \varphi(\underset{\lambda\in\Lambda}{\bigcap}\varphi^{-1}(\J_{\lambda}))=\underset{\lambda\in\Lambda}{\bigcap}\varphi(\varphi^{-1}(\J_{\lambda})), \]
namely
\[ \varphi(\underset{\lambda\in\Lambda}{\bigcap}\I_{\lambda})=\underset{\lambda\in\Lambda}{\bigcap}\varphi(\I_{\lambda}). \]
\end{proof}

\begin{cor}\label{CorRed}
For any Tambara functor $T$, if we define $T\red$ by
\[ T\red=T/\underset{\p\in\Spec(T)}{\bigcap}\p, \]
then $T\red$ becomes a reduced Tambara functor.
Obviously, $T$ is reduced if and only if $T=T\red$.

Moreover, the projection $T\rightarrow T\red$ induces a natural homeomorphism
\[ p_{\sharp}\colon\Spec(T)\rightarrow \Spec(T\red). \]
\end{cor}
\begin{proof}
By Corollary \ref{CorHomeoSpec},
\[ p_{\sharp}\colon V(\underset{\p\in\Spec(T)}{\bigcap}\p)\rightarrow \Spec(T\red) \ ;\ \p\mapsto p(\p) \]
is a bijection, and we have
\[ \underset{\q\in\Spec(T\red)}{\bigcap}\q=\underset{\p\in\Spec(T)}{\bigcap}p(\p)=p(\underset{\p\in\Spec(T)}{\bigcap}\p)=(0) \]
by Remark \ref{RemInt}.

The latter part also follows from Corollary \ref{CorHomeoSpec},
since we have
\[ V(\underset{\p\in\Spec(T)}{\bigcap}\p)=\Spec(T), \]
by Remark \ref{RemSpec}.
\end{proof}


\begin{lem}\label{LemSpecTopo}
Let $T,T_1,T_2$ be non-trivial Tambara functors on $G$, satisfying $T=T_1\times T_2$ as Tambara functors. Then any ideal of $T$ is of the form $\I_1\times\I_2$, where $\I_1\subseteq T_1$ and $\I_2\subseteq T_2$ are ideals.

Moreover, for any ideals $\I_1\subseteq T_1$ and $\I_2\subseteq T_2$, the ideal $\I_1\times \I_2\subseteq T$ is prime if and only if
\[ \I_1=T_1\ \text{and}\ \I_2\subseteq T_2\ \text{is prime} \]
or
\[ \I_1\subseteq T_1\ \text{is prime}\ \text{and}\ \I_2=T_2 \]
holds.
\end{lem}
\begin{proof}
The forepart immediately follows from the definition of an ideal.

To show the latter part, we use the criterion of Proposition \ref{PropPrimePrime}.
Suppose $\I_1\times\I_2$ is prime. Since
\[ (T_1\times(0))((0)\times T_2)=(0)\times(0)\subseteq \I_1\times\I_2, \]
it follows
\[ T_1\times (0)\subseteq \I_1\times\I_2\ \,\text{or}\ \, (0)\times T_2\subseteq \I_1\times\I_2, \]
namely, $\I_1=T_1$ or $\I_2=T_2$.

Suppose $\I_1=T_1$. Then, since
\[ (T_1\times T_2)/(T_1\times\I_2)\cong T_2/\I_2, \]
$T_1\times\I_2\subseteq T$ is prime if and only if $\I_2\subseteq T_2$ is prime.
\end{proof}

\begin{prop}\label{PropSpecTopo}
For any Tambara functor $T$, the following are equivalent.
\begin{enumerate}
\item $\Spec(T)$ is disconnected.
\item There exist coprime ideals $\I,\J\subseteq T\red$ such that $\I\cap\J=(0)$.
\item For any $X\in\Ob(\Gs)$, there exist $a,b\in T\red(X)$ satisfying $a+b=1$ and $\langle a\rangle\langle b\rangle=(0)$.
\item For some non-empty $X\in\Ob(\Gs)$, there exist $a,b\in T\red(X)$ satisfying $a+b=1$ and $\langle a\rangle\langle b\rangle=(0)$.
\item There exists a pair of non-trivial Tambara functors $T_1$ and $T_2$, and an isomorphism of Tambara functors $T\red\cong T_1\times T_2$.
\end{enumerate}
\end{prop}
\begin{proof}
By Corollary \ref{CorRed}, we may assume $T$ is reduced, $T\red=T$.

Suppose $\Spec(T)$ is disconnected. By definition of the topology on $\Spec(T)$, this is equivalent to the existence of ideals $\I,\J\subseteq T$ satisfying
\[ V(\I)\cup V(\J)=\Spec(T)\quad\text{and}\quad V(\I)\cap V(\J)=\emptyset. \]
By Remark \ref{RemSpec}, this is equivalent to
\[ \I\cap\J=(0)\quad\text{and}\quad\I+\J=T. \]
Thus {\rm (1)} is equivalent to {\rm (2)}.

If $T$ satisfies {\rm (2)}, then we have
\[ T=T/(\I\cap\J)\cong T/\I\times T/\J \]
by Corollary \ref{CorChinese}, and thus {\rm (5)} follows.
Besides, {\rm (5)} implies {\rm (1)} by Lemma \ref{LemSpecTopo}, since we have $\Spec(T)=V((0)\times T_2)\amalg V(T_1\times (0))\,(\approx \Spec(T_1)\amalg\Spec(T_2))$.
Obviously {\rm (3)} implies {\rm (4)}, and {\rm (4)} implies {\rm (2)}.
Thus it remains to show that {\rm (2)} implies {\rm (3)}.

If there exist ideals $\I,\J\subseteq T$ satisfying {\rm (2)}, then for any $X\in\Ob(\Gs)$, there exist $a\in\I(X)$ and $b\in\J(X)$ such that $a+b=1$ by Remark \ref{RemCoprime}
In particular $\langle a\rangle$ and $\langle b\rangle$ are coprime, and we have
\[ \langle a\rangle\langle b\rangle=\langle a\rangle\cap\langle b\rangle\subseteq\I\cap\J=(0), \]
and thus {\rm (3)} follows.
\end{proof}

We can take the radical of an ideal as follows.
\begin{prop}\label{PropRoot1}
Let $\I\subseteq T$ be an ideal. If we define a subset $\sqrt{\I}\subseteq T$ by
\[ \sqrt{\I}(X)=\{ a\in T(X)\mid \langle a\rangle^n\subseteq \I\ \ ({}^{\exists}n\in\mathbb{N}) \} \]
for each $X\in\Ob(\Gs)$, then $\sqrt{\I}$ becomes an ideal of $T$. We call this the {\it radical} of $\I$.
\end{prop}
\begin{proof}
First we show $\sqrt{\I}(X)\subseteq T(X)$ is an ideal for each $X\in\Ob(\Gs)$.

Take any $a,b\in\sqrt{\I}(X)$, and suppose
\[ \langle a\rangle^m\subseteq\I,\ \ \ \langle b\rangle^n\subseteq\I \]
holds for $m,n\in\mathbb{N}$.
Since $a+b\in(\langle a\rangle+\langle b\rangle)(X)$, we have
\[ \langle a+b\rangle^{m+n}\subseteq(\langle a\rangle+\langle b\rangle)^{m+n}=\sum_{i=0}^{m+n}\langle a\rangle^{m+n-i}\langle b\rangle^i\ \ \subseteq\I, \]
and thus $a+b\in\sqrt{\I}(X)$.
Besides, for any $r\in T(X)$, since $ra\in\langle a\rangle$, we have
\[ \langle ra\rangle^m\subseteq \langle a\rangle^m\subseteq\I, \]
and thus $ra\in\sqrt{\I}(X)$.

It remains to show $\sqrt{\I}$ is closed under $(\,)_+,(\,)_!,(\,)^{\ast}$. Take any $f\in\Gs(X,Y)$ and any $g\in\Gs(W,X)$.
Since we have
\begin{eqnarray*}
f_+(a)&\in&\langle a\rangle(Y),\\
f_!(a)&\in&\langle a\rangle(Y),\\
g^{\ast}(a)&\in&\langle a\rangle(W)
\end{eqnarray*}
by the definition of an ideal,
it follows
\begin{eqnarray*}
\langle f_+(a)\rangle^m&\subseteq&\langle a\rangle^m\ \subseteq\,\I,\\
\langle f_!(a)\rangle^m&\subseteq&\langle a\rangle^m\ \subseteq\,\I,\\
\langle g^{\ast}(a)\rangle^m&\subseteq&\langle a\rangle^m\ \subseteq\,\I,
\end{eqnarray*}
which means $f_+(a),f_!(a)\in\sqrt{\I}(Y)$ and $g^{\ast}(a)\in\sqrt{\I}(W)$.
\end{proof}

\begin{prop}\label{PropRoot2}
Let $T$ be a Tambara functor.
For any ideal $\I\subseteq T$, we have
\[ \sqrt{\I}\subseteq\underset{\p\in V(\I)}{\bigcap}\p. \]
In particular, $\sqrt{(0)}=(0)$ follows if $T$ is reduced.
\end{prop}
\begin{proof}
Let $X\in\Ob(\Gs)$ be any transitive $G$-set.
For any $a\in\sqrt{\I}(X)$, there exists $n\in\mathbb{N}$ such that $\langle a\rangle^n\subseteq \I$.
Thus for any $\p\in V(\I)$, we have
\[ \langle a\rangle^n\subseteq\I\subseteq\p.  \]
Since $\p$ is prime, this means $\langle a\rangle\subseteq\p$, namely $a\in\p(X)$. Thus we obtain $\sqrt{\I}(X)\subseteq\p(X)$.
\end{proof}

\begin{cor}
$ V(\sqrt{\I})=V(\I)$ holds for any ideal $\I\subseteq T$.
\end{cor}
\begin{proof}
By the definition of $\sqrt{\I}$, we have $\I\subseteq\sqrt{\I}$, which implies $V(\I)\supseteq V(\sqrt{\I})$. Conversely, Proposition \ref{PropRoot2} means $V(\I)\subseteq V(\sqrt{\I})$.
\end{proof}

\smallskip

\subsection{Monomorphic restriction condition}

Consider the following condition for a Tambara functor $T$, which we call the \lq{\it monomorphic restriction condition}'. This condition will be frequently used in the next subsection. 

\begin{dfn}\label{DefMRC}
A contravariant functor
\[ F\colon\Gs\rightarrow\Sett \]
is said to satisfy {\it monomorphic restriction condition} {\rm (MRC)} if it satisfies the following.
\begin{itemize}
\item[{\rm (MRC)}]\label{MRC} For any $f\in\Gs(X,Y)$ between transitive $X,Y\in\Ob(\Gs)$, the restriction $f^{\ast}$ is monomorphic.
\end{itemize}
Obviously, {\rm (MRC)} is equivalent to the following.
\begin{itemize}
\item[{\rm (MRC)}$\ppr$] For any transitive $Y\in\Ob(\Gs)$ and $\gamma\in\Gs(G/e,Y)$, the restriction $\gamma^{\ast}$ is monomorphic. (Moreover, this condition does not depend on the choice of $\gamma$.)
\end{itemize}

\medskip

\noindent Abbreviately, we say a Tambara functor $T$ {\it satisfies (MRC)} if $T^{\ast}$ satisfies {\rm (MRC)}.
\end{dfn}

\begin{ex}\label{ExMRC}
A Tambara functor $T$ satisfies {\rm (MRC)} for example in the following cases.
\begin{enumerate}
\item $T$ is a fixed point functor $\mathcal{P}_R$ associated to some $G$-ring $R$.
\item $T$ is additively cohomological, and $T(X)$ has no $|G|$-torsion 
for each transitive $X\in\Ob(\Gs)$. Here $|G|$ denotes the order of $G$.
\end{enumerate}
\end{ex}
\begin{proof}
{\rm (1)} follows from the definition of $\mathcal{P}_R$.
We show {\rm (2)}. Since $T$ is additively cohomological, for any $f\in\Gs(X,Y)$ between transitive $X,Y\in\Ob(\Gs)$, we have
\[ f_+f^{\ast}(y)=(\deg f)\, y\qquad({}^{\forall}y\in T(Y)). \]
Since $\deg f$ divides $|G|$ and $T(Y)$ has no $|G|$-torsion, this implies $f_+f^{\ast}$ is monomorphic. Especially $f^{\ast}$ is monomorphic.
\end{proof}

\begin{prop}\label{PropMRCFixed}
For any Tambara functor $T$ on $G$, the following are equivalent.
\begin{enumerate}
\item $T$ satisfies {\rm (MRC)}.
\item $T$ is a Tambara subfunctor $T\subseteq\mathcal{P}_R$ for some $G$-ring $R$.
\end{enumerate}
\end{prop}
\begin{proof}
Since $\mathcal{P}_R$ satisfies {\rm (MRC)} 
, obviously {\rm (2)} implies {\rm (1)}. To show the converse, we use the following.
\begin{lem}\label{LemMRCFixed}
$\ \ $
\begin{enumerate}
\item Let $M=(M^{\ast},M_{\ast})$ and $M\ppr=(M^{\prime\ast},M\ppr_{\ast})$ be two semi-Mackey functors on $G$, satisfying $M^{\ast}\subseteq M^{\prime\ast}$ as monoid-valued contravariant functors $\Gs\rightarrow\Mon$.\\
$($Namely, $M(X)\subseteq M\ppr(X)$ is a submonoid for each $X\in\Ob(\Gs)$ and satisfy $M^{\prime\ast}(f)|_{M(X)}=M^{\ast}(f)$ for any morphism $f\in\Gs(X,Y)$.$)$

If $M^{\prime\ast}$ $($and thus also $M^{\ast}$$)$ satisfies {\rm (MRC)}, then we have $M\subseteq M\ppr$ as semi-Mackey functors.

\item If $T$ and $T\ppr$ are Tambara functors satisfying {\rm (MRC)} and if there exists an inclusive morphism of ring-valued contravariant functors $\iota\colon T^{\ast}\hookrightarrow T^{\prime\ast}$,
then $\iota$ gives an inclusion of Tambara functors. Moreover if $\iota\colon T^{\ast}\overset{\cong}{\longrightarrow}T^{\prime\ast}$ is an isomorphism of ring-valued functors, then $\iota$ gives an isomorphism of Tambara functors,
\end{enumerate}
\end{lem}

\medskip

Suppose Lemma \ref{LemMRCFixed} is shown. Let $T$ be a Tambara functor satisfying {\rm (MRC)}. Endow $R=T(G/e)$ with the induced $G$-action, and take the fixed point functor $\mathcal{P}_R$ associated to $R$.

For any $H\le G$, if we let $p^H_e\colon G/e\rightarrow G/H$ be the canonical projection, then we have
\[ T(G/H)\overset{(p^H_e)^{\ast}}{\hookrightarrow} T(G/e) \]
by {\rm (MRC)}. Since this is $G$-equivariant and $H$ acts on $T(G/H)$ trivially, we obtain
\[ T(G/H)\overset{(p^H_e)^{\ast}}{\hookrightarrow} T(G/e)^H. \]
This means $(p^H_e)^{\ast}$ factors through $R^H=\mathcal{P}_{R}(G/H)$,
namely, there exists (uniquely) a ring homomorphism $\iota_H\colon T(G/H)\rightarrow\mathcal{P}_{R}(G/H)$ which makes the following diagram commutative.
\[
\xy
(-12,6)*+{T(G/e)}="0";
(12,6)*+{\mathcal{P}_R(G/e)}="2";
(-12,-6)*+{T(G/H)}="4";
(12,-6)*+{\mathcal{P}_R(G/H)}="6";
{\ar@{=} "0";"2"};
{\ar@{^(->}^{(p^H_e)^{\ast}} "4";"0"};
{\ar_{{}^{\exists}\iota_H} "4";"6"};
{\ar@{^(->}_{} "6";"2"};
{\ar@{}|\circlearrowright "0";"6"};
\endxy
\]

Thus $\iota_H$ becomes monomorphic by this commutativity.
It can be shown that $\{\iota_H\}_{H\le G}$ is compatible with conjugation maps. 
Besides, for any $K\le H\le G$, we obtain a commutative diagram
\[
\xy
(-12,6)*+{T(G/K)}="0";
(12,6)*+{\mathcal{P}_R(G/e)}="2";
(-12,-6)*+{T(G/H)}="4";
(12,-6)*+{\mathcal{P}_R(G/H)}="6";
(20,-7)*+{.}="7";
{\ar@{^(->}^{\iota_K} "0";"2"};
{\ar@{^(->}^{(p^H_K)^{\ast}} "4";"0"};
{\ar@{^(->}_{\iota_H} "4";"6"};
{\ar@{^(->}_{} "6";"2"};
{\ar@{}|\circlearrowright "0";"6"};
\endxy
\]

By the additivity of $T^{\ast}$ and $\mathcal{P}_R^{\ast}$, it turns out that $\{\iota_H\}_{H\le G}$ yields a injective natural transformation $\iota\colon T^{\ast}\hookrightarrow\mathcal{P}_R^{\ast}$ of ring-valued contravariant functors. 
By Lemma \ref{LemMRCFixed}, 
$\iota\colon T\hookrightarrow\mathcal{P}_R$ becomes a inclusive morphism of Tambara functors, and $T$ can be regarded as a Tambara subfunctor of $\mathcal{P}_R$ through $\iota$.

\bigskip

Thus it remains to show Lemma \ref{LemMRCFixed}.
\begin{proof}[Proof of Lemma \ref{LemMRCFixed}]
Remark that {\rm (2)} follows immediately if we apply {\rm (1)} to the additive and multiplicative parts respectively. Thus we only show {\rm (1)}.

It suffices to show
\[ M\ppr_{\ast}(f)|_{M(X)}=M_{\ast}(f) \]
for any morphism $f\in\Gs(X,Y)$. 
Furthermore, it is enough to show this for transitive $X,Y\in\Ob(\Gs)$.

Let $f\in\Gs(X,Y)$ be any morphism between transitive $X,Y\in\Ob(\Gs)$.
Let $\gamma_Y\in\Gs(G/e,Y)$ be a morphism, and take a pullback diagram
\begin{equation}\label{DiagPB}
\xy
(-8,6)*+{X\ppr}="0";
(8,6)*+{X}="2";
(-8,-6)*+{G/e}="4";
(8,-6)*+{Y}="6";
(11,-7)*+{.}="7";
(0,0)*+{\square}="8";
{\ar^{} "0";"2"};
{\ar_{} "0";"4"};
{\ar^{f} "2";"6"};
{\ar_{\gamma_Y} "4";"6"};
\endxy
\end{equation}
Since any point of $X\ppr$ has the stabilizer equal to $e$, we may assume $(\ref{DiagPB})$ is of the form
\[
\xy
(-14,8)*+{\underset{1\le i\le n}{\amalg} G/e}="0";
(14,8)*+{X}="2";
(-14,-8)*+{G/e}="4";
(14,-8)*+{Y}="6";
(0,0)*+{\square}="8";
{\ar^{\underset{1\le i\le n}{\cup}\zeta_i} "0";"2"};
{\ar_{\nabla} "0";"4"};
{\ar^{f} "2";"6"};
{\ar_{\gamma_Y} "4";"6"};
\endxy
\]
for some $\zeta_1,\ldots,\zeta_n\in\Gs(G/e,X)$.

By the Mackey condition, for any $x\in M(X)\ (\subseteq M\ppr(X))$, we have
\begin{eqnarray*}
M^{\prime\ast}(\gamma_Y)M_{\ast}(f)(x)&=&M^{\ast}(\gamma_Y)M_{\ast}(f)(x)\\
&=&\sum_{1\le i\le n}M^{\ast}(\zeta_i)(x)\\
&=&\sum_{1\le i\le n}M^{\prime\ast}(\zeta_i)(x)\\
&=&M^{\prime\ast}(\gamma_Y)M\ppr_{\ast}(f)(x).
\end{eqnarray*}
Since $M^{\prime\ast}(\gamma_Y)$ is monomorphic, we obtain
\[ M_{\ast}(f)(x)=M\ppr_{\ast}(f)(x)\quad({}^{\forall}x\in M(X)),  \]
which means $M_{\ast}(f)=M\ppr_{\ast}(f)|_{M(X)}$.
\end{proof}
\end{proof}

As in the proof of Lemma \ref{LemMRCFixed}, any Tambara functor $T$ satisfying {\rm (MRC)} becomes a Tambara subfunctor of $\mathcal{P}_{T(G/e)}$. It will be a natural question whether, for a ring $R$, there exists a non-trivial Tambara subfunctor $T\subseteq\mathcal{P}_R$ satisfying $T(G/e)=R$.
The following gives some criterion and a counterexample.
\begin{ex}
$\ \ $
\begin{enumerate}
\item Let $\mathbb{F}$ be a field satisfying $(\mathrm{char}(\mathbb{F}),|G|)=1$.
If a Tambara subfunctor $T\subseteq \mathcal{P}_{\mathbb{F}}$ satisfies $T(G/e)=\mathbb{F}$, then it satisfies $T=\mathcal{P}_{\mathbb{F}}$.
\item Let $G=\mathbb{Z}/p\mathbb{Z}$ for some prime number $p$, let $k$ be a field with characteristic $p$, and let $R=k[\mathbf{X}]$ be a polynomial ring, on which $G$ acts trivially. Then $\mathcal{P}_R$ admits a Tambara subfunctor $T$ defined by $T(G/e)=R$ and $T(G/G)=R^p=\{ f^p\mid f\in R \}$.
\end{enumerate}
\end{ex}
\begin{proof}
\begin{enumerate}
\item It suffices to show $T(G/H)=\mathcal{P}_{\mathbb{F}}(G/H)$ holds for any $H\le G$.
Let $x\in \mathcal{P}_{\mathbb{F}}(G/H)=\mathbb{F}^H$ be any element. Remark that we have
\begin{equation}\label{EqFix}
(p^H_e)_+(p^H_e)^{\ast}(x)=|H|\cdot x,
\end{equation}
where $|H|$ is the order of $H$. Since $(\mathrm{char}(\mathbb{F}),|G|)=1$ by assumption, 
we can easily show $\frac{1}{|H|}x\in\mathcal{P}_{\mathbb{F}}(G/H)$. Applying $(\ref{EqFix})$ to $\frac{1}{|H|}x$, we obtain
\[ (p^H_e)_+(p^H_e)^{\ast}(\frac{1}{|H|}x)=|H|\cdot\frac{1}{|H|}x=x. \]

Since $T$ is closed under $\mathcal{P}_{\mathbb{F}+}$ and $T(G/e)=\mathbb{F}\ni\frac{1}{|H|}(p^H_e)^{\ast}(x)
$, it follows that
\[ x=(p^H_e)_+(\frac{1}{|H|}(p^H_e)^{\ast}(x))\ \in T(G/H). \]
\item Since $G$ acts on $R$ trivially, it suffices to show $T$ is closed under $\mathcal{P}_R^{\ast}(p^G_e)$, $\mathcal{P}_{R+}(p^G_e)$, and $\mathcal{P}_{R\bullet}(p^G_e)$.

Since $T(G/e)=R$, obviously we have $\mathcal{P}_R^{\ast}(p^G_e)(T(G/G))\subseteq T(G/e)$.
Moreover, since
\[ (\mathcal{P}_{R+}(p^G_e))(f)=pf=0,\quad (\mathcal{P}_{R\bullet}(p^G_e))(f)=f^p \]
for any $f\in R$, we obtain
\begin{eqnarray*}
\mathcal{P}_{R+}(p^G_e)(T(G/e))&\subseteq&T(G/G)\\
\mathcal{P}_{R\bullet}(p^G_e)(T(G/e))&\subseteq&T(G/G).
\end{eqnarray*}
\end{enumerate}
\end{proof}

Let $T$ be any Tambara functor. As in Proposition \ref{PropIM}, there exists an ideal
\[ \I^T_{(0)}=\I_{(0)}\subseteq T \]
associated to the zero-ideal $(0)\subseteq T(G/e)$.
If we define $\TMRC\in\Ob(\TamG)$ by $\TMRC=T/\I^T_{(0)}$, then $\TMRC$ gives the \lq {\it canonical {\rm (MRC)}-zation}' of $T$.
More precisely, we have the following.
\begin{thm}\label{ThmMRCUniv}
Let $\TamM\subseteq\TamG$ denote the full subcategory of Tambara functors satisfying {\rm (MRC)}. Then the following holds.
\begin{enumerate}
\item For any $T\in\Ob(\TamG)$, we have $\TMRC\in\Ob(\TamM)$.
\item The correspondence $T\mapsto\TMRC$ gives a functor
\[ (\ )_{\mathrm{MRC}}\colon\TamG\rightarrow\TamM. \]
\item The functor in {\rm (2)} is left adjoint to the inclusion $\TamM\hookrightarrow\TamG$.
\end{enumerate}
\end{thm}
\begin{proof}
\begin{enumerate}
\item For any $G$-invariant ideal $I\subseteq T(G/e)$, the quotient $T/\I_I$ satisfies {\rm (MRC)}. In fact, for any transitive $X\in\Ob(\Gs)$ and $\gamma\in\Gs(G/e,X)$, we have
\[ ((T/\I_I)^{\ast}(\gamma))^{-1}(0)=(T^{\ast}(\gamma))^{-1}(I)/(\I_I(X))=\I_I(X)/\I_I(X)=0. \]
\item Let $\varphi\colon T\rightarrow S$ be a morphism in $\TamG$. For any $X\in\Ob(\Gs)$ and any $\gamma\in\Gs(G/e,X)$, we have a commutative diagram
\[
\xy
(-12,7)*+{T(G/e)}="0";
(12,7)*+{S(G/e)}="2";
(-12,-7)*+{T(X)}="4";
(12,-7)*+{S(X)}="6";
(17,-8)*+{.}="7";
{\ar^{\varphi_{G/e}} "0";"2"};
{\ar^{\gamma^{\ast}} "4";"0"};
{\ar_{\gamma^{\ast}} "6";"2"};
{\ar_{\varphi_Y} "4";"6"};
{\ar@{}|\circlearrowright "0";"6"};
\endxy
\]
Thus we have
\begin{eqnarray*}
\I^T_{(0)}(X)&=&\underset{\gamma\in\Gs(G/e,X)}{\bigcap}(\gamma^{\ast})^{-1}(0)\\
&\subseteq&\underset{\gamma\in\Gs(G/e,X)}{\bigcap}\varphi_X^{-1}((\gamma^{\ast})^{-1}(0))\ =\ \varphi_X^{-1}(\I^S_{(0)}(X)),
\end{eqnarray*}
which implies $\I^T_{(0)}\subseteq\varphi^{-1}(\I^S_{(0)})$.
Thus there exists uniquely
\[ \bar{\varphi}\in\TamG(\TMRC,\SMRC)\ \ (=\TamM(\TMRC,\SMRC)) \]
which makes the following diagram commutative,
\[
\xy
(-9,6)*+{T}="0";
(9,6)*+{S}="2";
(-9,-6)*+{\TMRC}="4";
(9,-6)*+{\SMRC}="6";
{\ar^{\varphi} "0";"2"};
{\ar_{p^T} "0";"4"};
{\ar^{p^S} "2";"6"};
{\ar_{\bar{\varphi}} "4";"6"};
{\ar@{}|\circlearrowright "0";"6"};
\endxy
\]
where $p^T$ and $p^S$ denote the canonical projections of Tambara functors. By the uniqueness, immediately we obtain a functor $(\ )_{\mathrm{MRC}}\colon\TamG\rightarrow\TamM$.

\item This also follows immediately from the above argument. In fact, for any $T\in\Ob(\TamG)$ and any $S\in\Ob(\TamM)$, we have a natural bijection
\[ -\circ p^T\colon \TamM(\TMRC,S)\overset{\simeq}{\longrightarrow}\TamG(T,S). \]
\end{enumerate}
\end{proof}

\begin{ex}\label{ExMRCZ}
$\ \ $
\begin{enumerate}
\item If $T$ satisfies {\rm (MRC)}, then naturally we have $T=\TMRC$.
\item $\Omega_{\mathrm{MRC}}\cong\mathcal{P}_{\mathbb{Z}}$, where $G$ acts trivially on $\mathbb{Z}$ in the right hand side.
\end{enumerate}
\end{ex}
\begin{proof}
{\rm (1)} is trivial. We show {\rm (2)}.
First we remark the following.
\begin{rem}\label{RemOmega}
For each $H\le G$, let $\mathcal{O}(H)$ denote a set of representatives of conjugacy classes of subgroups of $H$. $\mathcal{O}(H)$ has a partial order $\po$, defined by 
\[ K\po L\quad\Leftrightarrow\quad K\le L^h\ \text{for some}\ h\in H \]
for $K,L\in\mathcal\mathcal{O}(H)$. If $K\po L$ and $K\ne L^h$ for any $h\in H$, we write $K\pn L$.

$\Omega(G/H)$ is a free module over
\[ \{ G/K=(G/K\overset{p^H_K}{\rightarrow}G/H) \mid K\in\mathcal{O}(H) \}, \]
where $p^H_K\colon G/K\rightarrow G/H$ is the canonical projection.

Especially, for any transitive $X\cong G/H\in\Ob(\Gs)$, any $a\in\Omega(X)$ can be decomposed uniquely as
\begin{equation}\label{Eqa}
a=\sum_{K\in\mathcal{O}(H)} m_K\, G/K\qquad(m_K\in\mathbb{Z}).
\end{equation}
\end{rem}



We construct an isomorphism of Tambara functors
\[ \wp\colon\Omega_{\mathrm{MRC}} \cong\mathcal{P}_{\mathbb{Z}}. \]
Let $\wp_e\colon\Omega_{\mathrm{MRC}}(G/e)\overset{\cong}{\longrightarrow}\mathbb{Z}$ be the ring isomorphism given by
\[ \wp_e(m\, (G/e\rightarrow G/e))=m\quad ({}^{\forall} m\in\mathbb{Z}). \]
For any $H\le G$, put $\wp_H=\wp_e\circ(p^H_e)^{\ast}$. This is a ring homomorphism. 
\[
\xy
(32,0)*+{\mathbb{Z}}="0";
(0,18)*+{\Omega_{\mathrm{MRC}}\, (G/G)}="2";
(-12,6)*+{\Omega_{\mathrm{MRC}}\, (G/H)}="4";
(-12,-6)*+{\Omega_{\mathrm{MRC}}\, (G/K)}="6";
(0,-18)*+{\Omega_{\mathrm{MRC}}\, (G/e)}="8";
(-30,19)*+{}="3";
(-40,0)*+{}="5";
(-30,-19)*+{}="7";
{\ar^{\wp_G} "2";"0"};
{\ar^{\wp_H} "4";"0"};
{\ar_{\wp_K} "6";"0"};
{\ar_{\wp_e} "8";"0"};
{\ar_<<<<<{(p^G_H)^{\ast}} "2";"4"};
{\ar_{(p^H_K)^{\ast}} "4";"6"};
{\ar_>>>>>{(p^K_e)^{\ast}} "6";"8"};
{\ar@{}|\circlearrowright "0";"3"};
{\ar@{}|\circlearrowright "0";"5"};
{\ar@{}|\circlearrowright "0";"7"};
\endxy
\quad({}^{\forall}K\le {}^{\forall}H\le G)
\]

Remark that $(p^H_e)^{\ast}\colon\Omega_{\mathrm{MRC}}(G/H)\rightarrow\Omega_{\mathrm{MRC}}(G/e)$ is monomorphic by Theorem \ref{ThmMRCUniv}. Thus $\wp_H$ is monomorphic for any $H\le G$.
On the other hand, for any $m\in\mathbb{Z}$ we have
\[ \wp_H(m(H/H\overset{\id}{\rightarrow}H/H))=\wp_e(m(e/e\overset{\id}{\rightarrow}e/e))=m, \]
and thus $\wp_H$ becomes surjective.
Thus $\wp_H$ is an isomorphism for each $H\le G$.
As a corollary $(p^H_K)^{\ast}$ becomes an isomorphism, and we obtain a commutative diagram
\[
\xy
(-12,6)*+{\Omega_{\mathrm{MRC}}\, (G/H)}="0";
(12,6)*+{\mathbb{Z}}="2";
(-12,-6)*+{\Omega_{\mathrm{MRC}}\, (G/K)}="4";
(12,-6)*+{\mathbb{Z}}="6";
(15,-7)*+{.}="7";
{\ar^<<<<<{\wp_H}_<<<<<{\cong} "0";"2"};
{\ar_{(p^H_K)^{\ast}} "0";"4"};
{\ar^{\id_{\mathbb{Z}}} "2";"6"};
{\ar_<<<<<{\wp_K}^<<<<<{\cong} "4";"6"};
{\ar@{}|\circlearrowright "0";"6"};
\endxy
\]

It can be easily shown that $\{ \wp_H \}_{H\le G}$ is also compatible with conjugation maps, and thus form an isomorphism
\[ \wp\colon\Omega_{\mathrm{MRC}}^{\ast}\overset{\cong}{\longrightarrow}\mathcal{P}_{\mathbb{Z}}^{\ast} \]
of ring-valued contravariant functors $\Gs\rightarrow \Ring$.
By Lemma \ref{LemMRCFixed}, $\wp$ gives an isomorphism of Tambara functors
$\wp\colon \Omega_{\mathrm{MRC}}\overset{\cong}{\longrightarrow}\mathcal{P}_{\mathbb{Z}}$.
\end{proof}

\begin{ex}\label{ExNot}
As we have seen in the proof of {\rm (1)} in Theorem \ref{ThmMRCUniv}, if an ideal $\I\subseteq T$ satisfies $T/\I\in\!\!\!\!\!/\ \Ob(\TamM)$, then $\I$ can not be written as $\I=\I_I$ for any $G$-invariant ideal $I\subseteq T(G/e)$.

Consequently, if $\varphi\colon T\rightarrow S$ is a morphism of Tambara functors and if $S$ does not satisfy {\rm (MRC)}, then $\Ker\,\varphi$ is not of the form $\I=\I_I$, since we have $T/\Ker\,\varphi\cong\mathit{Im}\,\varphi\subseteq S$ by Remark \ref{RemHom}.

For example, if we use the Tambarization functor (\cite{N_TamMack}, \cite{N_DressPolyHopf}), we can construct such an example. For any semi-Mackey functor $M$ on $G$, if we denote its Tambarization by $\Omega [M]$, then by the functoriality of the Tambarization, 
\[
\xy
(-14,0)*+{0}="0";
(0,-4)*+{}="1";
(0,4)*+{M}="2";
(14,0)*+{0}="4";
{\ar^{0} "0";"2"};
{\ar^{0} "2";"4"};
{\ar@/_0.80pc/_{\id_0} "0";"4"};
{\ar@{}|\circlearrowright "1";"2"};
\endxy
\]
yields a commutative diagram of Tambara functors
\[
\xy
(-14,0)*+{\Omega}="0";
(0,-4)*+{}="1";
(0,4)*+{\Omega [M]}="2";
(14,0)*+{\Omega}="4";
(17,-1)*+{.}="5";
{\ar "0";"2"};
{\ar^{p} "2";"4"};
{\ar@/_0.80pc/_{\id_{\Omega}} "0";"4"};
{\ar@{}|\circlearrowright "1";"2"};
\endxy
\]
Since $\Omega$ does not satisfy {\rm (MRC)}, $\Ker\, p$ is not of the form $\I_I$. Also remark that we have $\Omega [M]/\Ker\, p\cong \Omega$.
\end{ex}

\smallskip

\subsection{Domains and fields}
In the commutative ring theory, a ring $R$ is an integral domain (resp. a field) if and only if $(0)\subsetneq R$ is prime (resp. maximal).
In a similar way we define a \lq domain-like' (resp. \lq field-like') Tambara functor, and show some analogous properties.

\begin{dfn}\label{DefDF}
Let $T$ be a Tambara functor.
\begin{enumerate}
\item $T$ is {\it domain-like} if $(0)\subsetneq T$ is prime.
\item $T$ is {\it field-like} if $(0)\subsetneq T$ is maximal.
\end{enumerate}
\end{dfn}

As a corollary of the arguments so far, we have:
\begin{cor}\label{CorDF}
Let $T$ be a Tambara functor, and $\I\subseteq T$ be an ideal.
\begin{enumerate}
\item $\I\subsetneq T$ is a maximal ideal if and only if $T/\I$ is field-like.
\item $\I\subsetneq T$ is a prime ideal if and only if $T/\I$ is domain-like.
\item If $T$ is field-like, then $T$ is domain-like.
\item If $T$ is domain-like, then $T$ is reduced.
\end{enumerate}
\end{cor}
\begin{proof}
{\rm (1)} follows from Proposition \ref{Prop1to1}.
{\rm (2)} follows from Theorem \ref{ThmMorphSpec}.
{\rm (3)} follows from Proposition \ref{PropPrimeMax}.
{\rm (4)} is trivial.
\end{proof}

\begin{cor}\label{CorDomain}
For any Tambara functor $T$, the following are equivalent.
\begin{enumerate}
\item $T$ is domain-like.
\item For any transitive $X,Y\in\Ob(\Gs)$ and any non-zero $a\in T(X), b\in T(Y)$, there exists a pair of diagrams
\[ C\overset{v}{\leftarrow}D\overset{w}{\rightarrow}X\quad\text{and}\quad C\overset{v\ppr}{\leftarrow}D\ppr\overset{w\ppr}{\rightarrow}Y, \]
such that
\[ (v_!w^{\ast}(a))\cdot(v\ppr_!w^{\prime\ast}(b))\ne0. \]
\end{enumerate}
\end{cor}
\begin{proof}
This immediately follows from Corollary \ref{CorPrimeDef}.
\end{proof}

\begin{prop}\label{PropDF2}
If $T$ satisfies {\rm (MRC)} and $T(G/e)$ is an integral domain, then $T$ is domain-like.
\end{prop}
\begin{proof}
We use the criterion of Corollary \ref{CorDomain}. Let $a\in T(X)$ and $b\in T(Y)$ be any pair of elements, with transitive $X,Y\in\Ob(\Gs)$. Suppose
\[ (v_!w^{\ast}(a))\cdot(v\ppr_!w^{\prime\ast}(b))=0 \]
is satisfied for any $C\overset{v}{\leftarrow}D\overset{w}{\rightarrow}X$, $C\overset{v\ppr}{\leftarrow}D\ppr\overset{w\ppr}{\rightarrow}Y$.
Especially for morphisms $\gamma\colon G/e\rightarrow X$ and $\gamma\ppr\colon G/e\rightarrow Y$, we have $\gamma^{\ast}(a)\gamma^{\prime\ast}(b)=0$. Since $T(G/e)$ is an integral domain, this is equivalent to
\[ \gamma^{\ast}(a)=0\quad\text{or}\quad\gamma^{\prime\ast}(b)=0, \]
which is, in turn, equivalent to
\[ a=0\quad\text{or}\quad b=0. \]
\end{proof}

\begin{thm}\label{ThmField}
For any Tambara functor $T\ne 0$, the following are equivalent.
\begin{enumerate}
\item $T$ is field-like.
\item $T$ satisfies {\rm (MRC)}, and $T(G/e)$ has no non-trivial $G$-invariant ideal.
\end{enumerate}
\end{thm}
\begin{proof}
First we show {\rm (1)} implies {\rm (MRC)}.

Suppose $T$ does not satisfy {\rm (MRC)$^{\prime}$}.
Then there exist transitive $X\in\Ob(\Gs)$ and non-zero $a\in T(X)$ such that
\[ \gamma_0^{\ast}(a)=0 \]
holds for some $\gamma_0\in\Gs(G/e,X)$. Since any $G$-map $\gamma\in\Gs(G/e,X)$ can be written as $\gamma=\gamma_0\circ f$ for some $f\in\Gs(G/e,G/e)$, it follows that
\begin{equation}\label{Eqf}
\gamma^{\ast}(a)=f^{\ast}\gamma_0^{\ast}(a)=0\quad({}^{\forall}\gamma\in\Gs(G/e,X)).
\end{equation}
Recall that, by Proposition \ref{PropPrincipal}, we have
\[ \langle a\rangle(G/e)=\{ p_+(c\cdot(v_!w^{\ast}(a)))\mid G/e\overset{p}{\leftarrow}C\overset{v}{\leftarrow}D\overset{w}{\rightarrow}A,\ c\in T(C) \}. \]
In this notation, since any point in $D$ has stabilizer equal to $\{ e\}$, $(\ref{Eqf})$ implies $w^{\ast}(a)=0$. Consequently we have $\langle a\rangle(G/e)=0$, and thus
\[ (0)\subsetneq\langle a\rangle\subsetneq T, \]
which means $(0)\subseteq T$ is not maximal. 

Second, we show {\rm (1)} implies that $T(G/e)$ has no non-trivial ideal. Suppose $T(G/e)$ has a non-trivial ideal $0\ne I\subsetneq T(G/e)$. By Proposition \ref{PropIM}, there is an ideal $\I_I\subseteq T$ satisfying $\I_I(G/e)=I$. Thus $(0)\subseteq T$ is not a maximal ideal.

\smallskip

Conversely, we show {\rm (2)} implies {\rm (1)}. It suffices to show, for any transitive $X\in\Ob(\Gs)$, any non-zero element $a\in T(X)$ satisfies $\langle a\rangle=T$.

Take any $\gamma\in\Gs(G/e,X)$. By {\rm (MRC)} we have 
$0\ne\gamma^{\ast}(a)\in\langle a\rangle(G/e)$.
Since $\langle a\rangle(G/e)$ is a $G$-invariant ideal of $T(G/e)$, this implies $\langle a\rangle(G/e)=T(G/e)$. By Proposition \ref{PropTrivIdeal}, it follows $\langle a\rangle=T$.
\end{proof}

Theorem \ref{ThmField} also can be shown using the following.
\begin{rem}\label{RemField}
For any Tambara functor $T$, the set of maximal ideals of $T$ is bijective to the set of $G$-invariant ideals of $T(G/e)$ maximal with respect to the inclusion.
\end{rem}
\begin{proof}
This immediately follows from Proposition \ref{PropIM}.

\end{proof}


\begin{ex}\label{ExDF}
$\ \ $
\begin{enumerate}
\item If $\mathbb{F}$ is a $G$-ring where $\mathbb{F}$ is a field, then $\mathcal{P}_{\mathbb{F}}$ is field-like.
\item If $R$ is a $G$-ring where $R$ is an integral domain, then $\mathcal{P}_R$ is domain-like.
\end{enumerate}
\end{ex}

However, even if $T$ is field-like, $T(G/e)$ is not necessarily an integral domain as in the following example.
\begin{ex}\label{ExCounterDF}
Let $G=\mathfrak{S}_n$ be the $n$-th symmetric group.
Let $\mathbb{F}$ be a field, and let $R$ be the product of $n$-copies of $\mathbb{F}$, 
on which $G$ acts by
\begin{eqnarray*}
\sigma\cdot (x_1,\ldots,x_n)=(x_{\sigma(1)},\ldots,x_{\sigma(n)}),
\end{eqnarray*}
for any $\sigma\in\mathfrak{S}_n$ and $(x_1,\ldots,x_n)\in R$.
Then $\mathcal{P}_R$ 
becomes a field-like Tambara functor by Theorem \ref{ThmField}.
\end{ex}

\medskip

On the other hand, we have the following.
\begin{rem}\label{RemFixedField}
For any field-like Tambara functor $T$, the following holds.
\begin{enumerate}
\item $T(G/e)^G$ is a field.
\item $T(G/G)$ is a subfield of $T(G/e)^G$.
\end{enumerate}
\end{rem}
\begin{proof}
{\rm (1)} For any $G$-invariant $x\in T(G/e)$, we have
\[ \langle x\rangle (G/e)=\{ rx\mid r\in T(G/e) \} \]
by Proposition \ref{PropPrincipal}.
If $x\ne 0$, we have $\langle x\rangle (G/e)\ni 1$, and thus there exists some $r\in T(G/e)$ satisfying $rx=1$.

{\rm (2)} Remark that $T$ satisfies {\rm (MRC)}, and $(p^G_e)^{\ast}\colon T(G/G)\hookrightarrow T(G/e)^G$ is injective. Since $T$ is a Tambara subfunctor of $\mathcal{P}_{T(G/e)}$ by Proposition \ref{PropMRCFixed}, we have
\begin{equation}
\label{Eqpp}
(p^G_e)_{\bullet}(p^G_e)^{\ast}(x)=x^{|G|}
\end{equation}
for any $x\in T(G/e)$. Since $(p^G_e)^{\ast}(x)$ is $G$-invariant, it is invertible in $T(G/e)$ by {\rm (1)} if $x\ne 0$. By $(\ref{Eqpp})$, it follows that $x\ne 0$ is invertible in $T(G/G)$.
\end{proof}

\smallskip

In the rest, we show $\Omega\in\Ob(\TamG)$ is a domain-like Tambara functor, for any finite group $G$.

We use the notation in Remark \ref{RemOmega}.
Recall that for any transitive $X\cong G/H\in\Ob(\Gs)$, any $a\in\Omega(X)$ can be decomposed uniquely as
\begin{eqnarray*}
a&=&\sum_{K\in\mathcal{O}(H)} m_K\, G/K\qquad(m_K\in\mathbb{Z})\\
&=&mX+\sum_{K\pn H} m_K\, G/K\qquad(m_K\in\mathbb{Z},\ m=m_H)
\end{eqnarray*}
in $\Omega(X)$. We denote this integer $m$ by $\rho_X(a)$.
This defines a ring homomorphism
\[ \rho_X\colon\Omega(X)\rightarrow\mathbb{Z}. \]



\begin{prop}\label{Proprho}
Let $X\in\Ob(\Gs)$ be transitive. For any $a\in\Omega(X)$, we have
\[ \rho_{G/G}((\pt_X)_!(a))=\rho_X(a). \]
\end{prop}
\begin{proof}
We use the following lemmas.

\begin{lem}\label{LemForProprho1}
Suppose $a=(A\overset{p}{\rightarrow}X)$ where $A$ is a transitive $G$-set over $X$ satisfying $\Gs(X,A)=\emptyset$. Then we have
\[ \rho_{G/G}((\pt_X)_!(a))=0. \]
\end{lem}

\begin{lem}\label{LemForProprho2}
For any $a,b\in\Omega(X)$, we have
\[ \rho_{G/G}((\pt_X)_!(a+b))=\rho_{G/G}((\pt_X)_!(a))+\rho_{G/G}((\pt_X)_!(b)). \]
\end{lem}

\bigskip

Assume Lemma \ref{LemForProprho1} and Lemma \ref{LemForProprho2} are shown. Since
\[ \rho_{G/G}((\pt_X)_!(X))=\rho_{G/G}((\pt_X)_!(1))=1, \]
it follows
\[ \rho_{G/G}((\pt_X)_!(mX))
=m\ \ \quad({}^{\forall}m\in\mathbb{Z}) \]
by Lemma \ref{LemForProprho2}. Thus Proposition \ref{Proprho} hold for $a=mX$\ \ ($m\in\mathbb{Z}$).

By Remark \ref{RemOmega}, for any $X\cong G/H$, any $a\in\Omega(X)$ is decomposed as
\[ a=mX+\sum_{K\pn H} m_K\, G/K\qquad(m,m_K\in\mathbb{Z}). \]
Since $\Gs(X,G/K)=\emptyset$ for any $K\pn H$, we have
\[ \rho_{G/G}((\pt_X)_!(G/K))=0 \]
by Lemma \ref{LemForProprho1}.

Thus by Lemma \ref{LemForProprho2}, we obtain
\begin{eqnarray*}
\rho_{G/G}((\pt_X)_!(a))&=&\rho_{G/G}((\pt_X)_!(mX))+\sum_{K\pn H}m_K\, \rho_{G/G}((\pt_X)_!(G/K))\\
&=&m,
\end{eqnarray*}
and thus Proposition \ref{Proprho} follows.

\smallskip

Thus it remains to show Lemma \ref{LemForProprho1} and Lemma \ref{LemForProprho2}.

\begin{proof}[Proof of Lemma \ref{LemForProprho1}]
By the definition of $(\pt_X)_!$, we have
\begin{eqnarray*}
(\pt_X)_!(a)&=&\Pi_{\pt_X}(A)\\
&=&\{ \sigma\mid \sigma\colon X\rightarrow A\ \text{is a map of sets},\ p\circ\sigma=\id_X \}.
\end{eqnarray*}

If there exists $\sigma\in\Pi_{\pt_X}(A)$ such that $G_{\sigma}=G$, then it means $\sigma$ is a $G$-map from $X$ to $A$, which contradicts to $\Gs(X,A)=\emptyset$. Thus any $\sigma\in\Pi_{\pt_X}(A)$ does not have stabilizer equal to $G$, and thus $\rho_{G/G}((\pt_X)_!(a))=0$.
\end{proof}

\begin{proof}[Proof of Lemma \ref{LemForProprho2}]
In the notation of Remark \ref{RemExp} (by letting $f=\pt_X$), we have
\[ (\pt_X)_!(a+b)=(\pt_X)_{\bullet}(a+b)=s_+(t_{\bullet}r^{\ast}(a)\cdot t\ppr_{\bullet}r^{\prime\ast}(b)). \]

Remark that $V=G/G\times\{ C\mid C\subseteq X\}=\{ C\mid C\subseteq X\}$. Since $X$ is transitive, $C\in V$ has stabilizer equal to $G$ if and only if $C=\emptyset$ or $C=X$. Thus $V$ decomposes as
\[ V=\Ve\amalg\Vo\amalg\Vt, \]
where $\Ve=\{\emptyset\}$, $\Vt=\{ X\}$, and $\Vo$ is their complement $\Vo=V\setminus (\Ve\cup\Vt)$, which does not contain any point with stabilizer $G$. 

Under the identification $\Omega(V)\cong\Omega(\Ve)\times\Omega(\Vo)\times\Omega(\Vt)$,
we write any $v\in\Omega(V)$ as
\[ v=(v_{\mathfrak{e}},v_{\mathfrak{o}},v_{\mathfrak{t}}), \]
where $v_{\mathfrak{e}}\in\Omega(\Ve)$, $v_{\mathfrak{o}}\in\Omega(\Vo)$, $v_{\mathfrak{t}}\in\Omega(\Vt)$.
Looking at stabilizers, we can show
\begin{equation}
\label{Eqrho}
\rho_{G/G}(s_+(v))=\rho_{\Ve}(v_{\mathfrak{e}})+\rho_{\Vt}(v_{\mathfrak{t}}).
\end{equation}

Since $t(U)=\Vo\amalg\Vt$ and $t\ppr(U\ppr)=\Ve\amalg\Vo$, we have
\begin{eqnarray*}
t_{\bullet}r^{\ast}(a)&=&(1,(t_{\bullet}r^{\ast}(a))_{\mathfrak{o}},(t_{\bullet}r^{\ast}(a))_{\mathfrak{t}}),\\
t\ppr_{\bullet}t^{\prime\ast}(b)&=&((t\ppr_{\bullet}r^{\prime\ast}(b))_{\mathfrak{e}},(t\ppr_{\bullet}r^{\prime\ast}(b))_{\mathfrak{o}},1),
\end{eqnarray*}
and,
\begin{eqnarray*}
t_{\bullet}r^{\ast}(0)&=&(1,0,0),\\
t\ppr_{\bullet}r^{\prime\ast}(0)&=&(0,0,1).
\end{eqnarray*}
By $(\ref{Eqrho})$, we have
\begin{eqnarray}\label{*1}
\rho_{G/G}((\pt_X)_!(a+b))&=&\rho_{G/G}(s_+(t_{\bullet}r^{\ast}(a)\cdot t\ppr_{\bullet}r^{\prime\ast}(b)))\\
&=&\rho_{\Ve}((t\ppr_{\bullet}r^{\prime\ast}(b))_{\mathfrak{e}})+\rho_{\Vt}((t_{\bullet}r^{\ast}(a))_{\mathfrak{t}}).\nonumber
\end{eqnarray}
Especially, if we let $a=0$ or $b=0$, then we obtain respectively
\begin{eqnarray}
\rho_{G/G}((\pt_X)_!(b))&=&\rho_{\Ve}(t\ppr_{\bullet}r^{\prime\ast}(b)_{\mathfrak{e}}),\label{*2}\\
\rho_{G/G}((\pt_X)_!(a))&=&\rho_{\Vt}(t_{\bullet}r^{\ast}(a)_{\mathfrak{t}}).\label{*3}
\end{eqnarray}
By $(\ref{*1})$, $(\ref{*2})$ and $(\ref{*3})$, we obtain
\[ \rho_{G/G}((\pt_X)_!(a+b))=\rho_{G/G}((\pt_X)_!(a))+\rho_{G/G}((\pt_X)_!(b)). \]
\end{proof}
\end{proof}

\begin{thm}\label{ThmOmegaDom}
For any finite group $G$, the Burnside Tambara functor $\Omega$ on $G$ is domain-like.
\end{thm}
\begin{proof}
We use the following lemma.
\begin{lem}\label{LemForThm}
For any transitive $X\in\Ob(\Gs)$ and any non-zero $a\in T(X)$, there exist some transitive $\widetilde{X}_a$ and a $G$-map $\nu_a\in\Gs(\widetilde{X}_a,X)$ such that
\[ \rho_{\widetilde{X}_a}(\nu_a^{\ast}(a))\ne 0. \]
\end{lem}

\smallskip

Suppose Lemma \ref{LemForThm} is proven. Let $a\in T(X)$, $b\in T(Y)$ be arbitrary non-zero elements, with $X,Y\in\Ob(\Gs)$ transitive. By Lemma \ref{LemForThm}, there exist transitive $\widetilde{X}_a$, $\widetilde{Y}_b$ and $\nu_a\in\Gs(\widetilde{X}_a,X)$, $\nu_b\in\Gs(\widetilde{Y}_b,Y)$ such that
\[ \rho_{\widetilde{X}_a}(\nu_a^{\ast}(a))\ne 0,\quad \rho_{\widetilde{Y}_b}(\nu_b^{\ast}(b))\ne 0. \]

By Proposition \ref{Proprho}, we have
\[ \rho_{G/G}((\pt_{\widetilde{X}_a})_!\nu_a^{\ast}(a))\ne 0,\quad \rho_{G/G}((\pt_{\widetilde{Y}_b})_!\nu_b^{\ast}(b))\ne 0. \]
Since $\rho_{G/G}$ is a ring homomorphism, it follows
\[ \rho_{G/G}(\, ((\pt_{\widetilde{X}_a})_!\nu_a^{\ast}(a))\cdot((\pt_{\widetilde{Y}_b})_!\nu_b^{\ast}(b))\, )\ne 0. \]
In particular $((\pt_{\widetilde{X}_a})_!\nu_a^{\ast}(a))\cdot((\pt_{\widetilde{Y}_b})_!\nu_b^{\ast}(b))\ne 0$. By Corollary \ref{CorDomain}, this means $\Omega$ is domain-like.

\smallskip

Thus it remains to show Lemma \ref{LemForThm}.
\begin{proof}[Proof of Lemma \ref{LemForThm}]
We may assume $X=G/H$ for some $H\le G$. As in Remark \ref{RemOmega}, decompose $a$ as
\[ a=\sum_{K\in\mathcal{O}(H)}m_K\, G/K. \]
Remark that, for each $K,L\in\mathcal{O}(H)$, we have
\[ \rho_{G/K}((p^H_K)^{\ast}(G/L))\ne0\ \ \Leftrightarrow\ \ K\po L. \]

Thus, if we take a maximal element $K$ in the set associated to $a$
\[ S_a=\{ L\in\mathcal{O}(H)\mid m_L\ne 0 \}, \]
then this $K$ satisfies 
\[  m_K\ne 0\quad \text{and}\quad m_L=0\ \ (K\pn {}^{\forall}L\po H), \]
and we have
\[ \rho_{G/K}((p^H_K)^{\ast}(a))\ne0 \]
by the additivity of $(p^H_K)^{\ast}$ and $\rho_{G/K}$.
\end{proof}
\end{proof}

As a corollary of arguments so far, we obtain a strict inclusion $\Spec (\Omega(G/e))\subsetneq\Spec (\Omega)$ as follows. Combined with Theorem \ref{ThmOmegaDom}, we may conclude that the (prime) ideal theory discussed in this article is a more precise tool for a Tambara functor $T$, than the ordinary ideal theory for $T(G/e)$.

\begin{cor}
Let $G$ be a finite group. The correspondence
\[ \I_{(-)}\colon\Spec (\Omega(G/e))=\Spec (\mathbb{Z}) \rightarrow\Spec (\Omega)\ ;\ \p\mapsto\I_{\p} \]
gives a strict inclusion $\Spec (\Omega(G/e))\subsetneq\Spec (\Omega)$.
\end{cor}
\begin{proof}
As in Remark \ref{RemField}, $\I_{(-)}$ sends maximal ideals of $\Omega(G/e)$ to maximal ideals of $\Omega$ bijectively.
Besides, by Example \ref{ExMRCZ} and Proposition \ref{PropDF2}, we see that $\I_{(0)}\subseteq \Omega$ is also prime.

Thus we obtain a map
\[ \I_{(-)}\colon\Spec (\Omega(G/e)) \rightarrow\Spec (\Omega)\ ;\ \p\mapsto\I_{\p}, \]
which is obviously injective. 

Moreover, Theorem \ref{ThmOmegaDom} means $(0)\subseteq\Omega$ is a prime ideal. Since $\Omega$ does not satisfy {\rm (MRC)}, this ideal is not of the form $\I_I$. 
\end{proof}

\bigskip


\end{document}